\documentclass[a4paper,11pt,reqno]{article}
\usepackage[body=17cm,top=3cm,bottom=4cm]{geometry}
\usepackage[english]{babel}

\usepackage{amsrefs}
\usepackage{amsmath,amsfonts,amsthm,amscd,amssymb}
\usepackage[colorlinks=true, pdfstartview=FitV, linkcolor=blue, citecolor=blue, urlcolor=blue]{hyperref}
\usepackage{authblk}
\usepackage[notref,notcite]{showkeys}

\newtheorem{theorem}{Theorem}[section]
\newtheorem{lemma}[theorem]{Lemma}
\newtheorem{prop}[theorem]{Proposition}

\newtheorem{corro}[theorem]{Corollary}

\theoremstyle{definition}
\newtheorem{definition}[theorem]{Definition}

\theoremstyle{remark}
\newtheorem{remark}[theorem]{Remark}

\numberwithin{equation}{section}

\DeclareMathAlphabet{\mathsl}{OT1}{cmss}{m}{sl}
\SetMathAlphabet{\mathsl}{bold}{OT1}{cmss}{bx}{sl}

\usepackage{soul,xcolor}

\newcommand{\al}{\ensuremath{\alpha}}
\newcommand{\be}{\ensuremath{\beta}}
\newcommand{\ga}{\ensuremath{\gamma}}
\newcommand{\de}{\ensuremath{\delta}}
\newcommand{\ep}{\ensuremath{\epsilon}}

\newcommand{\la}{\ensuremath{\lambda}}

\newcommand{\si}{\ensuremath{\sigma}}

\newcommand{\om}{\ensuremath{\omega}}

\newcommand{\ve}{\ensuremath{\varepsilon}}

\newcommand{\vp}{\ensuremath{\varphi}}

\newcommand{\Ga}{\ensuremath{\Gamma}}
\newcommand{\De}{\ensuremath{\Delta}}

\newcommand{\La}{\ensuremath{\Lambda}}
\newcommand{\Si}{\ensuremath{\Sigma}}

\newcommand{\Om}{\ensuremath{\Omega}}
\newcommand{\cA}{\ensuremath{\mathcal A}}

\newcommand{\cC}{\ensuremath{\mathcal C}}

\newcommand{\cF}{\ensuremath{\mathcal F}}

\newcommand{\cO}{\ensuremath{\mathcal O}}
\newcommand{\cP}{\ensuremath{\mathcal P}}
\newcommand{\cQ}{\ensuremath{\mathcal Q}}

\newcommand{\cS}{\ensuremath{\mathcal S}}

\newcommand{\bbB}{\ensuremath{\mathbb B}}

\newcommand{\bbE}{\ensuremath{\mathbb E}}

\newcommand{\bbN}{\ensuremath{\mathbb N}} 
 
\newcommand{\bbP}{\ensuremath{\mathbb P}} 
 
\newcommand{\bbR}{\ensuremath{\mathbb R}}

\newcommand{\bbX}{\ensuremath{\mathbb X}} 
\newcommand{\bbY}{\ensuremath{\mathbb Y}}

\newcommand{\bfB}{\ensuremath{\mathbf B}}

\newcommand{\bfI}{\ensuremath{\mathbf I}}
\newcommand{\bfJ}{\ensuremath{\mathbf J}}
\newcommand{\bfK}{\ensuremath{\mathbf K}}
\newcommand{\bfL}{\ensuremath{\mathbf L}}

\newcommand{\bfP}{\ensuremath{\mathbf P}}

\newcommand{\bfX}{\ensuremath{\mathbf X}}
\newcommand{\bfY}{\ensuremath{\mathbf Y}}

\newcommand{\me}{\ensuremath{\mathrm{e}}}

\newcommand{\md}{\ensuremath{\mathrm{d}}}

\newcommand{\brb}{\bar{b}}
\newcommand{\diverg}{\mathrm{div}}

\DeclareMathOperator{\mean}{\mathbb{E}}
\DeclareMathOperator{\Mean}{\mathrm{E}}
\DeclareMathOperator{\prob}{\mathbb{P}}

\newcommand{\ldef}{\ensuremath{\mathrel{\mathop:}=}}
\newcommand{\rdef}{\ensuremath{=\mathrel{\mathop:}}}

\newtheorem{corollary}[theorem]{Corollary}
\begin{document}

%
%
\title{Enhanced Sanov and robust propagation of chaos}
\date{\today}
%
%
\author[1]{Jean-Dominique Deuschel}
\author[1,2]{Peter K. Friz}
\author[1,2]{Mario Maurelli}
\author[1]{Martin Slowik}
{
  \tiny
  \affil[1]{{Institut f\"ur Mathematik, Technische Universit\"at Berlin}}  \affil[2]{{Weierstra\ss --Institut f\"ur Angewandte Analysis und Stochastik, Berlin}}
}

\maketitle

\begin{abstract}
  We establish a Sanov type large deviation principle for an ensemble of interacting Brownian rough paths. As application a large deviations for the ($k$-layer, enhanced) empirical measure of weakly interacting diffusions is obtained. This in turn implies a propagation of chaos result in rough path spaces and allows for a robust subsequent analysis of the particle system and its McKean-Vlasov type limit, as shown in two corollaries.
\end{abstract}

\tableofcontents

\section{Introduction and main results}
\subsection{Large deviation and rough paths}
The present paper is concerned with the intersection of large deviations, rough paths and (weakly) interacting diffusions. We note (i) that large deviations have been one of the first application areas of rough paths theory: indeed, following Ledoux et al. \cite{LQZ02}, a large deviation principe for Brownian motion and L\'evy's area, scaled by $\epsilon$ and $\epsilon^2$ respectively, in rough path topology, will yield immediately the Freidlin--Wentzell theory of large deviations for diffusions with small noise --  its suffices to combine continuity of the It\^o-map in rough path sense with the contraction principle of large deviation theory; \cite{FrVi10}. (ii) The interplay of rough paths with interacting stochastic differential equations was pioneered in \cite{CL15}. This work, as well as the more recent \cite{Ba15}, required in particular the development of a McKean Vlasov theory in the context of random rough differential equations (which is not at all the aim of this paper). At last, (iii) large deviations for interacting diffusions is a huge field, a small selection of relevant references is given by \cite{DG87, dPdH96, Ta84, dMZ03, Sz91}.

In sense, we combine here aspects of all the afore-mentioned references. In particular, when compared to the many classical works (iii) an advantage of our approach is robustness: as soon as we have a LDP on a suitably enhanced space (``enhanced Sanov") -- on which most stochastic operations of interest are continuous, the raison d'\^etre of rough paths -- basic facts of large deviation theory, such as contraction principle or Varadhan's lemma become directly applicable. On the contrary, stronger versions of contraction principles or Varadhan lemma need suitable approximated continuity properties which must be checked case by case.

We briefly describe our main results.  Let $\{B^i : i \in \bbN\}$ be a family of independent $d$-dimensional standard Brownian motions,\footnote{Later on, we shall allow for non-trivial $Law(B^i_0) \equiv \lambda.$} on a fixed filtered probability space $(\Omega, \cA, (\cF_t)_t, \prob)$.  On a finite time-horizon, say $[0, T]$, we may regard them as $C([0,T]; \bbR^d)$-valued i.i.d.\ random variables.  By a classical law of large numbers (LLN) argument (e.g. \cite[Thm 11.4.1]{Du02}) the \emph{empirical measure}, $L_n$, a random measure on pathspace, converges to the $d$-dimensional Wiener measure $P^{\{d\}}$.  More precisely, with probability one,\footnote{We regard $L_n^B$ and $P^d$ as random variables with values in the (Polish) space $\mathcal{P}(C([0,T];R^d))$, equipped with the $C_b$-weak topology.}
\begin{align*}
  L_n^B (\om)
  \;\ldef\;
  \frac{1}{n}\, \sum^n_{i=1} \de_{ B^i(\om) }
  \;\longrightarrow\;
  P^{\{d\}} 
  \qquad \text{as} \quad n \to \infty,
\end{align*}
Sanov's theorem quantifies the speed of this convergence: for a measure $Q$ on $C([0,T]; \bbR^d)$,
\begin{align*}
  \bbP\big[ L_n^B (\om) \approx Q\big]
  \;\approx\;
  \exp\!\big( \!-\!n\, H(Q \,|\, P^{\{d\}})\big),              
\end{align*}
in the form of a large deviation principle \cite{DeSt89,DeZe10}, where $H$ is the relative entropy.  Now, for each $1 \leq i,j \leq n$, we introduce the $2d$-dimensional \emph{double-layer} process $B^{\{2\};ij} \equiv B^{ij}$ as
\begin{align}
  B^{ij}_t
  \;\ldef\;
  \big( B_t^i, B_t^j \big) \,\in\, \bbR^{2d},
\end{align}
and then define the \emph{enhanced double-layer} process $\bbB^{\{2\};ij} \equiv \bbB^{ij}$, with values in the space of $2d \times 2d$ matrices, as
\begin{align}
  \bbB^{ij}_t \;=\; \int^t_0 B^{ij}_s \otimes \circ\, \md B^{ij}_s,
\end{align}
where $\circ$ denotes Stratonovich integration.  Clearly, for any $i \ne j$, we have
\begin{align*}
  \mathop{\mathrm{Law}}(B^{ij})
  \;=\;
  \mathop{\mathrm{Law}}\big(B^{12}\big)
  \;=\;
  P^{\{2d\}},
\end{align*}
where $P^{\{e\}}$ denotes $e$-\emph{dimensional Wiener-measure}.  We are interested in (the $G^2(\bbR^{2d})$-valued process)
\begin{align*}
  \bfB^{\{2\};ij}
  \;\equiv\;
  \bfB^{ij}
  \;\equiv\;
  (B^{ij},\bbB^{ij})
  \;\ldef\; ((B^i, B^j), \bbB^{ij}),
\end{align*}
with law
\begin{align*}
  \mathop{\mathrm{Law}}({\bfB^{ij}})
  \;=\;
  \mathop{\mathrm{Law}}({\bfB^{12}})
  \;=\;
  \mathop{\mathrm{Law}}  (B^1,B^2, \bbB^{12})
  \;\rdef\;
  \bfP^{\{2d\}},
\end{align*}
where $\bfP$ denotes the \emph{enhanced ($e$-dimensional) Wiener-measure}.\footnote{That is, the law of $e$-dimension Brownian motion $B$ and all its iterated integrals of the form $\int B^k \circ \md B^l, \, 1 \le k,l \le e$.}  For every $n$, define the \emph{enhanced ``double-layer'' empirical measure} as
\begin{align}
  \bfL_n^{\bfB;\{2\}}(\om) \equiv  \bfL_n^{\bfB}(\om)
  \;\ldef\;
  \frac{1}{n^2}\, \sum^n_{i,j=1}\de_{(B^{ij}, \bbB^{ij})(\om)}
  \;\equiv\;
  \frac{1}{n^2}\sum^n_{i,j = 1} \de_{\, \bfB^{ij}(\om) }.
\end{align}
In order to extend the enhanced ``double-layer'' ($k=2$) empirical measure to any $k \geq 3$,  define the ($kd$-dimensional) \emph{$k$-layer process} $B^{\{k\};i_1, \ldots, i_k} \equiv (B^{i_1}, \ldots ,B^{i_k})$, its rough path lift  $\bfB^{\{k\};i_1, \ldots ,i_k}$, and then the \emph{enhanced ``$k$-layer'' empirical measure} given by $\bfL_n^{\bfB;\{k\}}(\om) \ldef n^{-k} \sum \de_{\, \bfB^{\{k\};i_1,...,i_k}(\om)} $ with summation over all $1 \leq i_1, \ldots ,i_k \leq n$.  One may expect, as suggested by our notation, that, for any integer $k$,\footnote{Again we regard $\bfL_n^{\bfB;\{k\}}$ and $\bfP^{\{kd\}}$ as random variables with values in the (Polish) space $\mathcal{P}(C([0,T];G^2(\bbR^{kd}))$, equipped with the $C_b$-weak topology.}
\begin{align}\label{eq:convEMRP}
  \bfL^{\bfB;\{k\}}_n(\om)
  \;\longrightarrow\;
  \bfP^{\{kd\}}
  \qquad \text{as} \quad
  n \to \infty.    
\end{align}
This is indeed the case, however not a consequence of LLN, for even when $k=2$ the $\{ \bfB^{ij}:  i, j = 1, \ldots, n\}$ are not independent. In fact, we shall study the speed of convergence around this limit: one of our main results is a large deviation principle for the law of $\bfL_n^{\bfB;\{k\}}$, which gives \eqref{eq:convEMRP}, with convergence in probability (and a.s. by a Borel-Cantelli argument) with respect to the $\al$-H\"older rough path topology, as a byproduct. Here and below $\cC_g^{0,\al}([0,T]; \bbR^{e})$ denotes a rough path space (in notation of \cite{FrHa14}, some recalls below), a Polish space, elements of which are paths with values in the group $G^{(2)}(\bbR^e) \subset \bbR^e \otimes (\bbR^e)^{\otimes 2}$.
\begin{theorem} \label{thm:main_res_intro}
  Fix $\al$ in $(1/3, 1/2)$.  The sequence of laws $\{\mathop{\mathrm{Law}}(\bfL_n^{\bfB;\{k\}}) : n \in \bbN\}$ satisfies a large deviation principle on $\cP(\cC_g^{0,\al}([0,T]; \bbR^{kd}))$ endowed with the $C_b$-weak topology, with scale $n$ and good rate function $\bfI\!:\cP(\cC_g^{0,\al}([0,T]; \bbR^{kd})) \to \bbR \cup \infty$ that is given by
  \begin{align}
    \bfI^{\{k\}}(\mu) \equiv  \bfI(\mu)
    \;=\;
    \begin{cases}
      H( \mu \circ \pi_1^{-1} \,|\, P^{\{d\}}), \quad
      &\text{if }
      \mu = F^{\{k\}}(\mu \circ \pi_1^{-1}),
      \\[1ex]
      +\infty,
      &\text{otherwise} .
    \end{cases}
  \end{align}
  This LDP is also valid in a stronger (``modified Wasserstein") topology. 
\end{theorem}
Here $\pi_1:  G^{(2)}(\bbR^{kd}) \cong  G^{(2)}( \oplus_{i=1}^k \bbR^{d}) \to \bbR^d$ is given by the projection $( (x^1, \dots, x^k ), \ldots ) \mapsto x^1$.  In particular, given a probabilty $\mu$ on $\cC_g^{0,\al}([0,T]; \bbR^{kd})$, the image measure $Q \ldef (\pi_1)_* \mu \equiv \mu \circ \pi_1^{-1}$ is a measure on the classical H\"older space $C^{0,\alpha}([0,T],\bbR^d)$.  Moreover, $H (. | P^{\{d\}})$ is the relative entropy and
\begin{align} \label{defFk}
  F^{\{k\}}\!: Q \;\longmapsto\; Q^{\otimes k}\circ (S^{\{kd\}})^{-1}
\end{align}
defines a map $F^{\{k\}}\!: \cP_1(C^{0,\al}([0,T];\bbR^d)) \rightarrow \cP_1(\cC^{0,\al}_g([0,T];\bbR^{kd}))$, where $S=S^{\{e\}}$ denotes the (measurable) lifting map of an $e$-dimensional path to a path with values in $G^2(\bbR^e)$.

The interest in a \emph{modified Wasserstein} topology (on probability measures on rough path space, Section 4 for details) is continuity of the map (here $k = 2$, but then trivially for $k \geq 2$ by projection)
\begin{align*}
  \mu
  \;\longmapsto\;
  \int_{\cC^{0,\al}_g([0,T];\bbR^{2d})}\, \int^T_0 \brb(X_t)\; \md\bfX_t\; \mu(\md\bfX)
\end{align*}
for sufficiently nice $\brb$.  Indeed, combining Girsanov's theorem and Varadhan's lemma will then imply a LDP for the empirical measures, as $n \to \infty$, for the particle system given by \footnote{Given a function $b\!: \bbR^{2d} \to \bbR^d$, we use the notation $(x^1, x^2) \in \bbR^d \times \bbR^d = \bbR^{2d}$ and we denote by $\brb\!:\bbR^{2d} \to \bbR^{2d}$ the function such that $\brb(x^1,x^2)^1 = b(x^1,x^2)$ and $\brb(x^1,x^2)^2 = 0$.}
\begin{align}
  \left\{
    \begin{array}{rcll}
      \md X^{i,n}_t
      &\mspace{-5mu}=\mspace{-5mu}
      &
      {\displaystyle
        \frac{1}{n}\,
        \sum_{j=1}^n b\big(X^{i,n}_t, X^{j,n}_t\big)\, \md t \,+\, \md B^i_t,
      }
      &\quad i = 1, \ldots n,\\[2ex]
      \mathop{\mathrm{Law}}(X^{i,n}_0)
      &\mspace{-5mu}=\mspace{-5mu}
      &\la
      \qquad \text{i.i.d.}
    \end{array}
  \right.
\end{align}
In fact, our approach not only  allows to recover the (known) rate function for the large deviations of such a particle system, of the form $J_b(Q) = H ( Q | \Phi (Q) )$ cf. Section \ref{rPOC} (where $\Phi = \Phi_b$ is introduced, such that fixed points of $\Phi$ are solutions to the martingale problem of the corresponding McKean--Vlasov equation with mean-field drift $b$), but it gives the LDP on the level of $k$-layer enhanced empirical measures. We shall see in two applications, namely Corollary \ref{intro:cor} and Corollary \ref{intro:cor_appl2} below, how useful exactly this can be. 
\begin{theorem}\label{intro:generalres}
  Assume that $b$ is in $C^2_b(\bbR^d \times \bbR^d)$, let $X^n = (X^{1,n}, \ldots, X^{n,n})$ be the solution to the above system where the initial law satisfies a suitable exponential integrability condition (Condition \eqref{eq:exp:intcond}).  Let $\bfL^{\bfX,\{k\}}_n$ be the corresponding enhanced $k$-layer empirical measure, $k \geq 2$.  Fix $\alpha \in (1/3,1/2)$.  Then the sequence of laws $\{\mathop{\mathrm{Law}}\big(\bfL^{\bfX,\{k\}}_n\big) : n \in \bbN \}$ satisfies a large deviation principle on (a modified Wasserstein) space of probability measures on $\cC^{0,\al}_g([0, T];\bbR^{kd})$ with scale $n$ and good rate function $\bfJ_b$ given by
  \begin{align}
    \bfJ_b^{\{k\}}(\mu) \equiv \bfJ_b(\mu)   
    \;=\;
    \begin{cases}
      H (\mu\circ\pi_1^{-1}| \Phi (\mu\circ\pi_1^{-1}) ),
      &\text{if }\; \mu=F^{\{k\}}(\mu\circ\pi_1^{-1}),\\[1ex]
      +\infty,
      &\text{otherwise}. 
    \end{cases}
  \end{align}
\end{theorem}

A first consequence of this large deviation principle, together with the fact that the rate function has a unique zero, is a ``law of large number'' which already contains a remarkably strong form of \emph{propagation of chaos} (POC), namely Theorem \ref{EPOC} below. Note that this result can also be recovered via classical It\^o calculus (the reader can verify this as an exercise), nevertheless it illustrates well the extra information carried by the LDP above, moreover its corollary \ref{intro:cor} is another example of the combination of mean field and rough paths arguments. For context, we first give the classical form of POC. Let us also note there is much new interest in POC, with recent applications ranging from calibration methods in quantitative finance to the analysis of lithium-ion batteries. 
\begin{theorem}[Classical POC, e.g. \cite{Sz91}] \label{thm:mfLDP_intro}
  Let $\{\bar{X}^{j} : j \in \bbN\} $ be an i.i.d. realizations of the McKean-Vlasov diffusion $\bar{X}$ (see Section \ref{rPOC} for details).  Then, for all $k \in \bbN$,
  \begin{align}
    \mathop{\mathrm{Law}}\big( X^{1,n}, \dots ,X^{k,n} \big) 
    \underset{n \to \infty }{\;\longrightarrow\;}
    \mathop{\mathrm{Law}}\big(\bar{X}^{1}, \dots, \bar{X}^{k}\big)
    \;=\;
    \mathop{\mathrm{Law}}( \bar{X} )^{\otimes k},
  \end{align}
  as $C_b$-weak\footnote{Actually, in $1$-Wasserstein sense ....} convergence of probability measures on $\big( C([0, T], \bbR^{d}) \big)^{\times k} \cong C([0, T], \bbR^{kd})$ equipped with uniform topology.
\end{theorem}
In classical terminology \cite{Sz91} the law of $\big( X^{1,n}, \dots, X^{k,n} \big)$ is \emph{$\cQ$-chaotic}, where $\cQ = \mathop{\mathrm{Law}}( \bar{X} )$ is a
probability measure on the (Polish) space $E = C([ 0, T], \bbR^{d})$.

We now state the enhanced POC\ on rough path space, that is paths with values in $G^{2}(\bbR^{N})$ rather than $\bbR^{N}$.  We insist that this is not just a form of the classical POC (a.k.a. $\cQ$-chaos) in which $E = C([0, T], \bbR^{d})$ is replaced by some other (Polish) space, which happens to be a rough path space.  To wit, the limiting measure in our Theorem \ref{EPOC} below is \emph{not} of product measure form, since it effectively tracks all areas between the particle trajectories (in the mean-field limit) which requires it to be a measure on the \emph{geometric rough path} space
\begin{align*}
  \cC^{0,\alpha}_g\big([0, T], \bbR^{kd}\big)
  \;\cong\;
  C^{0,\alpha}_g\big([0, T], G^{2}(\bbR^{kd}) \big) ,
\end{align*}
which indeed offers enough room to capture
\begin{align*}
  \int X^{j} \otimes \circ\, \md X^{j}...
  \qquad\text{for}\qquad
  1 \leq i,j \leq k
\end{align*}
(the anti-symmetric part of which corresponds to the afore-mentioned areas).  In contrast, a space of $k$ rough paths over $\bbR^{d}$, say
\begin{equation*}
  \big( \bfX^{1}, \dots, \bfX^{k} \big)
  \;\in\;
  \cC^{0, \alpha}_g\big([0, T], \bbR^{d}\big)^{\times k}
  \;\cong\;
  C^{0,\alpha }_g\big([0, T], \oplus_{i=1}^{k} G^{2}(\bbR^{d}) \big)
\end{equation*}
contains strictly less information as it contains, particle trajectories on $\bbR^{d}$ aside, only
\begin{align*}
  \int X^{i} \otimes \circ\, \md X^{i}
  \qquad \text{for} \qquad
  1 \leq i \leq k
\end{align*}
(and hence only the areas of each single $d$-dimensional particle trajectory).

This extra information contained in $\cC^{0,\alpha}_g\big([0, T], \bbR^{kd} \big)$ makes a difference indeed when one is interested in subsequent analysis of this particle system, as we shall see in the corollary below. But first we state our enhanced POC.  Recall that for a $e$-dimensional semimartingale $Z$, its Stratonovich (level $2$) lift is given by
\begin{align*}
  S^{\{e\}}(Z)
  \;=\;
  \Big(
    Z^{i} \,:\, 1 \leq i \leq e;
    \int Z^{i} \otimes \circ\, \md Z^{j} \,:\, 1 \leq i,j \leq e
  \Big).
\end{align*}
\begin{theorem}[Enhanced POC] \label{EPOC}
  Under the assumptions of the classical POC for all $k \in \bbN$
  \begin{align}
    \mathop{\mathrm{Law}}\big( S^{\{kd\}}( X^{1,n}, \dots, X^{k,n}) \big)    \underset{n \to \infty}{\;\longrightarrow\;}
    \mathop{\mathrm{Law}}\big( S^{\{kd\}}(\bar{X}^{1}, \dots ,\bar{X}^{k}) \big) ,
  \end{align}
  as $C_b$-weak convergence of probability measures on $\cC^{0,\alpha}_g([0, T], \bbR^{kd})$ equipped with $\alpha $-H\"{o}lder geometric rough path topology.
\end{theorem}
We now illustrate the power of this new form of propagation of chaos.  Recall that the solution flow to an SDE depends continuously on the driving noise in rough path topology (e.g. \cite{FrVi10}.) We then have immediately the following result, a direct proof of which would require
substantial work.
\begin{corro} \label{intro:cor}
  Fix some $k \in \bbN$ and consider, for $n \geq k$, the solution flow $Y^{n} \equiv Y$ to
  \begin{align}\label{intro:appl_SDE_1}
    \md Y_{t}
    \;=\;
    f_{0}(Y)\, \md t \,+\,
    \sum_{i=1}^{k} f_{i}(Y) \circ\, \md X^{i,n}
  \end{align}
  where the $f_{i}$'s are $C^{3}_b$ vector fields on $\bbR^{N}$ in the case $d=1$ (or $f_{i}\in C^{3}_b( \bbR^{N}; L(\bbR^{d},\bbR^{N})) $ more generally).  Then, (in the sense of flows, cf.\ \cite{FrVi10}, and $1/2^{-}$-H\"{o}lder on compacts in time)
  \begin{align}
    \mathop{\mathrm{Law}}(Y^{n})
    \underset{n \to \infty}{\;\longrightarrow\;}
    \mathop{\mathrm{Law}}(\bar{Y}),
  \end{align}
  where the weak limit flow is given by 
  \begin{align}
    \md \bar{Y}_{t}
    \;=\;
    f_{0}(\bar{Y})\, \md t \,+\,
    \sum_{i=1}^{k} f_{i}(\bar{Y}) \circ\, \md \bar{X}^{i}.
  \end{align}
\end{corro}

We give now a second application of Theorem \ref{intro:generalres}, which cannot be covered, to our understanding, by classical LDP results. This is a large deviation principle associated with SDEs driven by $k$-layer paths $(X^{i_1,n},\ldots X^{i_k,n})$: we take, for $i_1,\ldots i_k$ in $\{1,\ldots n\}$, the SDE
  \begin{align}\label{intro:appl_SDE_2}
    \md Y^{i_1,\ldots i_k;n}_{t}
    \;=\;
    \sum_{j=1}^{k} f_{i}(Y^{i_1,\ldots i_k;n}) \circ\, \md X^{i_j,n}
  \end{align}
with same initial condition $Y^{i_1,\ldots i_k;n}=y_0$; here $f_j:\bbR^d\rightarrow\bbR^m$, $j=1,\ldots k$, are given $C^3_b$ vector fields. For this SDE we can consider the empirical measure
\begin{align*}
L_n^{Y;\{k\}} = \frac{1}{n^k}\sum_{(i_1,\ldots i_k)\in\{1,\ldots n\}^k}\delta_{Y^{i_1,\ldots i_k;n}}.
\end{align*}
This empirical measure can be seen as a symmetrization of the system \eqref{intro:appl_SDE_1}, as it tracks the positions of $Y^{i_1,\ldots i_k}$ discarding the particular choice of indices $i_1,\ldots i_n$. Now, as in the previous application, rough paths provide continuity of the solution map of this SDE with respect to the driving noise in rough path topology. Therefore contraction principle implies the following:
\begin{corro}\label{intro:cor_appl2}
For any fixed $1/3<\beta<1/2$, the sequence $\{\text{Law}(L_n^{Y;\{k\}})|n\in\mathbb{N}\}$ satisfy a large deviation principle on $\cP(C^{0,\beta}([0,T];\bbR^m)$, endowed with the $C_0$-weak topology.
\end{corro}

The paper is organized as follows.  In Section 1, after a brief introduction, we explain our main results. Section~\ref{sec:not} is devoted to settle notation and some recalls on rough paths. In Section~\ref{section_enh_Sanov}, we prove the enhanced Sanov theorem (Theorem~\ref{thm:main_res_intro}) in the 1-Wasserstein metric, leaving the extension to the modified Wasserstein topology to Section \ref{section_modW}. For notational simplicity we focus in Sections~\ref{section_enh_Sanov} and \ref{section_modW} on the two-layer case ($k=2$).  We explain in Section~\ref{section:ext_k_layer} how to extend this to general $k$ (and so conclude with a full proof of Theorem~\ref{thm:main_res_intro}). In Section \ref{sec:LDP_wid}, we introduce the $n$-particle system, more precisely a system of $n$ weakly interacting diffusions, and prove a large deviation principle for the empirical measure, that is Theorem~\ref{thm:mfLDP_intro}. At last, in Sections \ref{rPOC} and \ref{appl2}, we prove resp.\ the (enhanced $k$-layer) propagation of chaos property and the LDP for the system \eqref{intro:appl_SDE_2}.

\subsection{Relation to the work of Cass--Lyons \cite{CL15}}
We comment in some detail on the relation of our work to Cass--Lyons. In \cite{CL15}, the authors first and foremost establish a theory of mean-field RDEs (more precisely, \cite[Theorem 4.9]{CL15}, rough differential equations with mean-field interaction in the drift term) for suitable classes of random rough paths $\bfB(\omega)$.  When it comes to propagation of chaos (see \cite[Section~5]{CL15}) they are able to consider interacting particle dynamics of the form
\begin{align}\label{CL_POC}
  \md \bfX^{i,n}(\omega)
  \;=\;
  \frac{1}{n}\,
  \sum_{j=1}^{n} b\big( X_{t}^{i,n}(\omega), X_{t}^{j,n}(\omega) \big)\, \md t
  \,+\, \si\big( X_{t}^{i}(\omega) \big)\, \md \bfB^{i}(\omega),  
\end{align}
with i.i.d.\ initial data and driving noise, $(X_{0}^{i,n}, \bfB^{i})$, and show (Theorem 5.2) that
\begin{align*}
  \bfL_{n}^{\bfX}
  \;\ldef\;
  \frac{1}{n}\, \sum_{j=1}^{n} \de_{\bfX^{i,n}}
  \;\longrightarrow\;
  \mathop{\mathrm{Law}}(\bar{X})
  \qquad \text{ a.s.}
\end{align*}
In the scale of $k$-layer enhanced empirical measure, $\bfL_{n}^{\bfX} \equiv \bfL_{n}^{\bfX;\{k\}}\big|_{k=1}$.  Furthermore, it is conjectured (see \cite[page~25]{CL15}) that their approach will be useful to establish Sanov-type theorem \`a la Dawson--G\"{a}rtner for \eqref{CL_POC}.  Although related, our work is \emph{not} a proof of this conjecture.  That said, such a result will not imply our main result, robust propagation of chaos (Theorem~\ref{EPOC}, and then e.g.\ Corollary~\ref{intro:cor}). To be more specific, in our work no mean-field RDE theory is required, and in fact we have taken the noise to be additive Brownian noise, that is $\md B^{i}(\omega)$ versus $\sigma( X_{t}^{i}(\omega))\, \md\bfB^{i}(\omega)$. (We note that including non-interacting\ diffusion coefficients $\md B^{i} \leadsto \sigma(X_{t}^{i})\, \md B^{i}$ would have been possible, as long as the Girsanov argument we use, cf.\ the proof of Theorem~\ref{generalres}, remains feasable, which amounts to an ellipticity assumption on $\sigma $.)  In the cases where our setting overlaps with \cite{CL15}, we indeed quantify the above with a large deviation principle, but then we also obtain (Theorem~\ref{intro:generalres}) a Sanov-type \`a la Dawson--G\"{a}rtner for the general $k$-layer enhanced empirical measure $\bfL_{n}^{\bfX;\{k\}}$.  This is in fact out of reach reach of \cite{CL15} as can be trivially seen noting that $\bfL_{n}^{\bfX;\{k\} }$ necessarily involves information of $( \bfX^{1,n}, \dots ,\bfX^{n,n})$, and hence (take e.g. $b\equiv 0$) of $(\bfB^{1}, \dots, \bfB^{n})$, as \emph{joint} rough path, rather then a collection of $n$ rough paths.  But no such information is assumed in \cite{CL15}, making $\bfL_{n}^{\bfX;\{k\} }$, $k \geq 2$, effectively an ill-defined object.  In contrast, for us, by working directly with Brownian motion, we always have the Stratonovich lift at our disposal, so this is not an issue.  For the same reason, our \emph{robust} propagation of chaos (Theorem~\ref{EPOC}, and then
e.g. Corollary~\ref{intro:cor}) can not possibly be obtained in the framework of
Cass--Lyons.  We finally note that forthcoming work of Bailleul--Catellier deals with Sanov-type theorem a la Dawson--G\"{a}rtner for \eqref{CL_POC}, again in the spirit of Cass--Lyons.
\medskip

\textbf{Acknowledgements:} P.F. and M.M. acknowledge funding from the European Research Council under the European UnionÕs Seventh Framework Program (FP7/2007-2013) / ERC grant agreement nr. 258237 and ECMath (Project SE8). P.F. und J.-D.D. acknowledge support from the DFG within Research Unit FOR 2402.

\section{Basic notation and results on rough paths}\label{sec:not}
We introduce the space of rough paths and the space where our empirical measures live. Most of this section is taken from \cite{FrHa14} or \cite{FrVi10}. Before going into the theory, let us recall the basis of $\alpha$-H\"older continuous functions. Given a Polish space $(E,d)$ with a compatible structure of Lie group (it will be $\bbR^e$ or $G^2(\bbR^e)$) and given $\al$ in $(0,1)$, we define the space $C^\al([0,T];E)$ of the $\al$-H\"older continuous paths from $[0,T]$ to $E$.  This is a complete metric space, endowed with the distance
\begin{align}
  d_{\al}(\ga^1, \ga^2)
  \;=\;
  \sup_{t \in [0,T]} d(\ga^1(t), \ga^2(t))
  \,+\,
  \sup_{s,t \in [0,T], s \neq t}
  \frac{d(\ga^1(s)^{-1} \ga^1(t), \ga^2(s)^{-1} \ga^2(t))}{|t-s|^\al}.
\end{align}
This space is not separable in general. However, the subspace $C^{0,\al}([0,T];E)$ given by the closure, with respect of $d_\al$, of the smooth ($C^\infty$) paths is separable, hence Polish. Furthermore, for any $\be > \al$, $C^\be([0,T]; E)$ is included in $C^{0,\al}([0,T];E)$ and the inclusion is compact.

When dealing with rough paths, we will always assume $\al$ in $(1/3, 1/2]$. An $\al$-H\"older rough path on $\bbR^e$ is a triple $\bfX = (X_0, X, \bbX)$, with $X_0$ point in $\bbR^e$, $X = (X_{s,t})_{s<t}$ two-index $\bbR^e$-valued map and $\bbX = (\bbX_{s,t})_{s<t}$ two-index $\bbR^{e\times e}$-valued map (we always suppose $0\le s,t\le T$ when not specified), satisfying the following conditions (here $v\otimes w$ denotes the tensor product $vw^T$):
\begin{enumerate}
\item algebraic conditions (Chen's relation): for any $s < u < t$,
  \begin{align}
    X_{s,t} \;=\; X_{s,u} \,+\, X_{u,t}
    \qquad \text{and} \qquad
    \bbX_{s,t} \;=\; \bbX_{s,u} \,+\, \bbX_{u,t} \,+\, X_{s,u} \otimes X_{u,t};
  \end{align}
\item analytic conditions:
  \begin{align}
    \sup_{0 \leq s < t \leq T} \frac{|X_{s,t}|}{|t-s|^\al} \;<\; \infty
    \qquad \text{and} \qquad
     \sup_{0 \leq s < t \leq T} \frac{|\bbX_{s,t}|}{|t-s|^{2\al}} \;<\; \infty.
  \end{align}
\end{enumerate}
Here $X_0$ represents the initial condition; it is not included in the standard definition (Definition 2.1 in \cite{FrHa14}, Chapter 2), but we need to keep track of it because we will work with paths starting from a generic probability measure (and not just from a single point). However, with some abuse of notation, we will usually write $\bfX=(X,\bbX)$, without $X_0$, when this is not relevant for our purposes, as for example when the initial point is fixed (this was the case for the main result \ref{thm:main_res_intro}).

The space of $\al$-H\"older rough paths on $\bbR^e$ is denoted by $\cC^{\al}([0, T]; \bbR^e)$. It is not a vector space (since the sum of two rough paths does not respect Chen's relation), but it is a complete metric space, endowed with the distance
\begin{align*}
  \tilde{\rho}_{\al}(\bfX, \bfY)
  \;=\;
  |X_0-Y_0| \,+\, \rho_{\al}(\bfX, \bfY)
  \;=\;
  |X_0-Y_0| \,+\,
  \sup_{0 \leq s < t \leq T} \frac{|X_{s,t}-Y_{s,t}|}{|t-s|^\al}
  \,+\, \sup_{0 \leq s < t \leq T} \frac{|\bbX_{s,t}-\bbY_{s,t}|}{|t-s|^{2\al}}.
\end{align*}
For convenience, we also introduce a ``norm'' on rough paths; this is actually not a norm, but it has some good homogeneity property. We define
\begin{align}
  \|\bfX\|_\al
  \;=\;
  \sup_{0 \leq s < t \leq T}\, \frac{|X_{s,t}|}{|t-s|^\al}
  \,+\,
  \sup_{0 \leq s < t \leq T}\, \frac{|\bbX_{s,t}|^{1/2}}{|t-s|^{\al}}.
\end{align}
A problem with the space $\cC^\al([0,T]; \bbR^e)$ is that it is not separable. That is why we introduce also the space $\cC^{0,\al}_g([0, T]; \bbR^e)$ of geometric rough paths. This is the subspace of $\cC^\al([0,T]; \bbR^e)$ obtained as the closure, with respect to the $\tilde{\rho}_\al$ distance, of the space of smooth $\bbR^e$-valued paths and their iterated integrals (see \cite{FrHa14}, Section 2.2). Now the space $\cC^{0,\al}_g([0, T];\bbR^e)$, endowed with the distance $\tilde{\rho}_{\al}$ is a Polish space. This will be the space of interest for us.

The space of geometric rough paths has also the following geometrical interpretation (taken for example from \cite[Section~2.3]{FrHa14}): it can be identified with the space $C^{0,\al}([0,T];G^2(\bbR^e)$ of the closure of smooth paths, with respect to the $\al$-H\"older topology, over a the (free step-$2$ nilpotent) Lie group $G^2(\bbR^e))$. In particular, we can consider the $\al$-H\"older distance $d_\al$ associated with the (Carnot-Caratheodory) distance in $G^2(\bbR^e))$, as explained at the beginning of this section, and we have, for a constant $C>0$,
\begin{align}\label{eq:equiv_dist_RP}
  C^{-1}(|X_0|+\|\bfX\|_\al)
  \leq
  d_{\al}(\bfX, 0)
  \leq 
  C(|X_0|+\|\bfX\|_\al)
\end{align}
for every geometric rough path $\bfX$.  We call this distance the homogeneous distance.  Unless otherwise stated, we will always use the homogeneous distance for geometric rough paths.  Notice however that, for the purpose of this paper, only the asymptotic behaviour of $d_{\al}(\bfX, 0)$, as $|X_0| + \|\bfX\|_{\al} \to \infty$, is of interest for us (see Sections \ref{W_section} and \ref{mod_W_subsection} on the link between this behaviour and the Wasserstein topology), therefore one can use $|X_0|+\|\bfX\|_\al$ instead of $d_\al(\bfX,0)$.

A consequence of this geometrical interpretation is that, for any $\alpha<\beta$, we have the continuous embedding for rough path spaces,
\begin{align}
  \cC^{\be}_g([0,T]; \bbR^e)
  \;\hookrightarrow\;
  \cC^{0,\al}_g([0,T]; \bbR^e)
  \;\hookrightarrow\;
  \cC^{\al}_g([0,T]; \bbR^e),
\end{align}
where the first embedding is compact.

A basic result in Lyons' rough paths theory is that, given a function $f$ regular enough, the integral $\int^t_0 f(Y)\, \md \bfY$ is well defined and continuous with respect to $Y$ in the rough paths topology.  We have (e.g. Theorem 4.4 in \cite{FrHa14}, Chapter 4):
\begin{theorem}\label{thm:cont_rough_int}
  Let $f$ be a function in $C^2_b(\bbR^e)$ and let $\bfX$ be a geometric $\al$-H\"older rough path on $\bbR^e$. Given a partition $\De$ of the interval $[0,T]$, define the approximated integral on $\De$ as
  \begin{align}
    I_{\De} f(\bfX)
    \;=\;
    \sum_{[s,t] \in \De} f(X_{s}) X_{s,t} \,+\, Df(X_{s})\, \bbX_{s,t}.
  \end{align}
  Then, the limit
  \begin{align}
    \int^T_0 f(X)\, \md\bfX
    \;\ldef\;
    \lim_{n \to \infty} I_{\De_n}f(X)
  \end{align}
  exists for every sequence $(\De_n: n\in \bbN)$ with infinitesimal size $|\De_n| = \sup_{[s,t] \in \De_n}(t-s)$ and is independent of the sequence itself.  Furthermore, the application
  \begin{align}
    \cC^{0,\al}_g([0,T];\bbR^e) \;\ni\; \bfX
    \;\mapsto\;
    \int^T_0 f(X)\,\md\bfX \;\in\; \bbR
  \end{align}
  is continuous and it holds, for some constant $C_f$ depending on $f$,
  \begin{align}
    \left|\int^T_0 f(X)\, \md\bfX\right|
    \;\leq\;
    C_f\, \big(\|\bfX\|_\al \vee \|\bfX\|_\al^{1/\al} \big).
  \end{align}
\end{theorem}
Recall that Theorem~\ref{thm:main_res_intro}, through definition of $F^{\{k\}}$ given in (\ref{defFk}), involves a (measurable) ``rough path lifting map''
\begin{align}
  S \equiv S^{\{e\}} \!: C^{0,\al}([0,T],\bbR^e ) \;\rightarrow\; \cC^{0,\al}_g([0,T], \bbR^e) .
\end{align}
Here is the precise definition.  Consider piecewise linear approximation $\{X^k\ : k \in \bbN\}$ of $X$ based on dyadic partitions $\De_k \ldef \{[2^{-k}i, 2^{-k}(i+1)] \,:\, i \in \bbN\}$, $k \in \bbN$, and set
\begin{align}
  A^k_{s,t} \;\equiv\; \lim_{k \to \infty}\, \int^t_s X^k_{s,r} \otimes \md X^k_r.
\end{align}
Whenever $S^k \ldef (X^k,A^k)$ is Cauchy in $\al$-H\"older rough path metric, set
\begin{align}\label{eq:def:S}
  S(X) \;\equiv\; (X,A(X)) \;\ldef\; \lim_{k \to \infty} (X^k,A^k)
\end{align}
and zero elsewhere.  By construction, $S(X)$ is in $\cC^\al$ and actually in $\cC^{0,\al}_g$ (since $X^k$ is a Lipschitz path and so $S^k$ is in $\cC^{0,\al}_g$) and $X \mapsto S(X)$ is a well-defined measurable (but in general discontinuous!) map on path space. \\

We now recall the basic relations between rough and stochastic integration, see \cite[Proposition 3.5, 3.6 and Corollary 5.2]{FrHa14}.  We allow $B$ to start from a generic initial probability measure $\lambda$ with finite second moment.
\begin{prop}\label{prop:conv_S}
  Let $B$ be an $e$-dimensional standard Brownian motion over a filtered probability space $(\Omega, (\cF_t)_t, \prob)$, with initial measure $\la$ with finite second moment.  For any $i, j = 1, \ldots, e$, let $(\bbB^{\mathrm{Strat}}_{s,t})^{ij} = \int^t_0 B^i_{s,r} \circ \md B^j_r$ be its Stratonovich iterated integral.  Then,
  \begin{enumerate}
  \item[(i)] $\prob$-a.s.,\ $\bfB^{\mathrm{Strat}} \ldef (B, \bbB^{\mathrm{Strat}})$ is a geometric $\al$-H\"older rough path for any $\al < 1/2$;  \item[(ii)] there exists a null-set $N$ with respect to the $e$-dimensional Wiener measure $P = P^{\{e\}}$ (and hence to every $Q$ absolutely continuous with respect to $P^{\{e\}}$), such that, away from this null-set, $S^k$ is a Cauchy sequence in the rough path metric and $S(B) = \bfB = (B, \bbB)$ $\prob$-a.s.
  \end{enumerate}
\end{prop}
\begin{prop}\label{rough_Strat_int}
  Let $B$ be as before and let $f$ be a function in $C^2_b(\bbR^e)$. Then the Stratonovich integral
  \begin{align}
    \int^T_0f(B) \circ \md B
  \end{align}
  and the rough integral
  \begin{align}
    \int^T_0 f(B)\, \md\bfB^{\mathrm{Strat}}
  \end{align}
  coincide $\bbP$-a.s..
\end{prop}

\section{The enhanced Sanov theorem}\label{section_enh_Sanov} 
The main objective in this section is to prove an LDP for the enhanced empirical measures $\bfL^\bfB_n = \bfL^{\bfB;\{k\}}_n$ in the $1$-Wasserstein topology, in the double layer case ($k=2$ will be fixed, and often omitted, throughout this section).  For this purpose, consider a sequence of independent $d$-dimensional Brownian motions $\{B^i : i \in \bbN\}$ each starting with initial distribution $\la$, defined on some filtered probability space $(\Om, \cA, (\cF_t)_t, \prob)$. In the sequel, for fixed $\alpha \in (1/3, 1/2)$, we use the convention to denote a generic measure on $C^{0,\al}([0,T];\bbR^d)$ by $Q$, and we write $\prob^Y$ to denote the law on this space of a process $Y$; $P^{\{d\}}=\bbP^B$ is the Wiener measure on $C^{0,\al}([0,T];\bbR^d)$ with initial distribution $\lambda$ unless differently specified.

The empirical measure $L^B_n$ is defined as
\begin{align}
  L^B_n \;=\; \frac{1}{n} \sum^n_{i=1} \de_{B_i}.
\end{align}
We use the $1$-Wasserstein metric as the topology on the space of probability measures (with finite first moment) on the spaces $C^\alpha$ and $\cC^{0,\al}_g$. In this topology, all the maps of the form
\begin{align}
  \mu \;\mapsto\; \int \varphi\; \md\mu,
\end{align}
for $\varphi$ continuous with at most linear growth, are continuous; on the contrary, in the $C_b$-weak topology we could only allow for continuous bounded $\varphi$. The reason why we consider the $1$-Wasserstein metric is mainly because it is more convenient in the proof: first it gives an easy-to-handle distance between probability measure, then it makes the map $C^{0,\al}(\bbR^e)\ni\mu \mapsto \int\bbX^{(m)}\mu(\md X)$ (where $\bbX^{(m)}$ will be a suitable approximation of the stochastic integral $\int^t_0X_r\otimes\circ\md X_r$) continuous for $m$ fixed ($X\mapsto\bbX^{(m)}$ has linear growth with respect to $d_\al$, so the $C_b$-weak topology would not fit into this scheme).

The Section is organized as follows.  We start with proving Sanov theorem in the $1$-Wasserstein metric. Then, as an intermediate result, we prove an LDP for the double-layer empirical measures which is a consequence of Sanov theorem (in $1$-Wasserstein metric) and the contraction principle. Finally, we show an LDP for the enhanced empirical measures, whose proof uses the idea for the double-layer empirical measures but exploits the extended contraction principle, together with approximation lemmata coming from rough paths theory.

\subsection{Sanov theorem in $1$-Wasserstein metric}
We quickly review Sanov theorem in $1$-Wasserstein metric on a general Polish space. A necessary and sufficient condition for Sanov theorem in $p$-Wasserstein metric was in fact given in \cite{WWW10}, but as the argument is short we include it in a form convenient to us.

Given a Polish space $(E, d_E)$, we denote by $\cP_1(E)$ the space of probability measures on $E$ with finite first moment, i.e.\ the probability measures $\mu$ satisfies $\int_E d_E(x,x_0)\, \mu(\md x) <+\infty$ for some (equivalently for all) $x_0\in E$. It is a Polish space endowed with the $1$-Wasserstein distance $d_{W}$, namely
\begin{align}\label{eq:def:Wasserstein}
  d_{W}(\mu,\nu)
  \;=\;
  \inf_{\pi \in \Ga(\mu,\nu)}
  \left\{\int_{E\times E}d_E(x^1,x^2)\, \pi\big(\md(x^1, x^2)\big)\right\},
\end{align}
where $\Ga(\mu, \nu)$ is the set of all probability measures on $E \times E$ with the first marginal and the second marginal equal resp.\ to $\mu$ and $\nu$. Whenever $E$ is some (Polish) space of $\alpha$-H\"older continuous (rough) paths, cf.\ beginning of Section \ref{sec:not}, we shall write $d_{W,\alpha}$ for the corresponding $1$-Wasserstein distance. Some basic facts on $1$-Wasserstein metric will be specified later in the Appendix.

We also recall that the relative entropy between two probability measures $\mu$ and $\nu$ on $F$ is defined as
\begin{align}
  H(\mu | \nu)
  \;=\;
  \begin{cases}
    \int_F\rho\log\rho\md\nu, \quad
    &\text{if }\; \mu\ll\nu\;\text{ and }\; \frac{\md\mu}{\md\nu}=\rho,    \\[1ex]
    +\infty,
    &\text{otherwise}. 
  \end{cases}
\end{align}
\begin{theorem}[Sanov theorem in Wasserstein metric]\label{SanovWasserstein}
  Let $E$ be a Polish space and let $(X^i)_i$ be a sequence of $E$-valued i.i.d.\ random variables, with law $\mu$. Assume that $\mu$ satisfies the following condition: there exists a function $G:E\rightarrow[0,+\infty]$, with compact sublevel sets (in particular lower semi-continuous), with more than linear growth (i.e.,\ for some $x_0$, $|G(x)|/d(x,x_0)\rightarrow+\infty$ as $d(x,x_0)\rightarrow+\infty$), such that
  \begin{align}\label{eq:morethanlin}
    \int_E \me^{G}\, \md\mu \;<\; +\infty.
  \end{align}
  Then the sequence of laws of the empirical measures
  \begin{align}
    L^X_n \;=\; \frac{1}{n} \sum^n_{i=1} \de_{X_i}
  \end{align}
  satisfies a large deviation principle on $\cP_1(E)$, endowed with the $1$-Wasserstein metric, with rate $n$ and good rate function $H(\cdot|\mu)$.
\end{theorem}
This result differs from the classical Sanov theorem by the fact that it involves the $1$-Wasserstein metric, while classical Sanov theorem involves $C_b$-weak topology. In this, the statement above is stronger, but does need the additional condition on the measure $\lambda$.
\begin{remark}\label{SanovWiener}
  In the case $E = C^{0,\al}([0,T];\bbR^d)$, $\al < 1/2$, the assumption above is satisfies by $\{B^i : i \in \bbN\}$ (independent Brownian motions starting from $\la$), if $\la$ verifies Condition \eqref{eq:exp:intcond}.  Indeed one can take
  \begin{align}
    G(\ga)
    \;=\;
    c \bigg(
      \sup_{0 \leq s < t \leq T} \frac{d(\gamma(t),\gamma(s))}{|t-s|^\beta}
    \bigg)^{\!1+\ve} \,+\, c |\ga(0)|^{1+\ve},
  \end{align}
  where $\beta$ is in $(\al, 1/2)$ and $c$, $\ve$ are the same of Condition \eqref{eq:exp:intcond}. This $G$ has compact sublevel sets and more than linear growth; Condition \eqref{eq:morethanlin} is verified since ($B^1(x=0)$ is the Brownian motion starting at $0$)
  \begin{align}
    \mean\!\big[\me^{G(B^1)}\big]
    \;=\;
    \mean\!\big[\exp(c\|B^1(x=0)\|_{C^\be}^{1+\ve})\big]\,
    \int_{\bbR^d}\me^{c|x|^{1+\ve}}\la(dx)
    \;<\;
    \infty,
  \end{align}
  by Condition \eqref{eq:exp:intcond} and exponential integrability of $c\|B^1\|_{C^\be}^{1+\ve}$ (a consequence for example of Corollary 13.15 in \cite{FrVi10}).
\end{remark}
\begin{proof}[Proof of Theorem \ref{SanovWasserstein}]
  The assertion is a consequence of classical Sanov theorem (in the weak convergence topology, see for example \cite[Theorem~3.2.17]{DeSt89}) and the inverse contraction principle, see \cite[Theorem~4.2.4]{DeZe10}, provided we prove exponential tightness, in $1$-Wasserstein metric, of the laws of the empirical measures $L^X_n$. We need to prove that, for any $M>0$, there exists a compact set $K = K_M$ in $\cP_1(E)$ (with the $1$-Wasserstein metric) such that
  \begin{align}\label{exptight}
    \limsup_n\frac{1}{n}\log\mu_n[K_M^c] \;<\; -M.
  \end{align}
  We take $K_M$ as in Lemma \ref{criterion_compactness}.  By Markov inequality and i.i.d.\ hypothesis on $X^i$, for any $C_M$, we have
  \begin{align}
    \prob\!\big[L^X_n \in K_M^c\big]
    \;\leq\;
    \me^{-nC_M}\, \mean\!\bigg[\exp\Big(\int_E n G\, \md L^X_n\Big)\bigg]    \;=\;
    \me^{-nC_M}\, \mean\!\big[\exp\!\big(G(X_1)\big)\big]^n.
  \end{align}
  The assumption implies that $A \ldef \mean[\exp(G(X_1))] < \infty$.  Hence, by taking $C_M = M +\log A + 1$, we obtain \eqref{exptight} which completes the proof.
\end{proof}

\subsection{The LDP for the double-layer empirical measure}
As a warm-up example, we investigate what happens with the double layer empirical measure
\begin{align}
  L^{B,\{2\}}_n
  \;=\;
  \frac{1}{n^2}\sum^{n}_{i,j=1}\delta_{(B^i,B^j)}
  \;\in\;
  \cP_1(C^{0,\al}([0,T];\bbR^{2d})),
\end{align}
where $\cP_1(C^{0,\al}([0,T];\bbR^{2d}))$ denotes the space of probability measures on $C^{0,\al}([0,T];\bbR^{2d})$ endowed with the $1$-Wasserstein metric. In the following, we identify $C^{0,\al}([0,T];\bbR^{2d})$ with $C^{0,\al}([0,T];\bbR^d)^2$ (they are equivalent as metric spaces) and we call $\pi_1$ the canonical projection in $C^{0,\al}([0,T];\bbR^d)^2$ on the first $d$ components.
\begin{lemma}
  The double layer empirical measure $L^{B,\{2\}}_n$ is the image of the empirical measure $L^B_n$ under the map $Q \mapsto Q \otimes Q$.
\end{lemma}
\begin{proof}
  Obvious via the identification $\de_{(B^i,B^j)} = \de_{B^i} \otimes \de_{B^j}$.
\end{proof}
\begin{prop}
  The family $\{\mathop{\mathrm{Law}}(L^{B,\{2\}}_n) : n \in \bbN\}$ satisfies a LDP on $\cP_1(C^{0,\al}([0,T];\bbR^{2d}))$ endowed with the $1$-Wasserstein metric, with scale $n$ and good rate function $I^{\{2\}}$, given by
  \begin{align}
    I(Q^{\{2\}})
    \;=\;
    \begin{cases}
      H( Q \,|\, P^{\{d\}}), \quad
      &\text{if } Q^{\{2\}} = (Q \otimes Q)\text{  with } Q=Q^{\{2\}}\circ\pi_1^{-1}
      \\[1ex] \infty,
      &\text{otherwise}. 
    \end{cases}
  \end{align}
\end{prop}
\begin{proof}
  The result is a consequence of Sanov theorem in the $1$-Wasserstein metric \ref{SanovWasserstein} (together with Remark \ref{SanovWiener} for our context) and the contraction principle, cf.\ \cite[Theorem~4.2.1]{DeZe10}, provided that the map
  \begin{align}
    \cP_1(C^{0,\al}([0,T];\bbR^d))
    \;\ni\;
    Q
    \;\longmapsto\;
    Q\otimes Q
    \;\in\;
    \cP_1(C^{0,\al}([0,T];\bbR^{2d}))
  \end{align}
  is continuous.  This continuity result is provided in Lemma~\ref{continuity_doubling} in the Appendix.
\end{proof}

\subsection{The LDP for the enhanced empirical measure}
We are ready to prove the large deviation result for sequence of the enhanced empirical measure $\{\bfL^{\bfB}_n : n \in \bbN\}$.

Probability measures on $\cC^{0,\al}_g([0,T];\bbR^{2d})$ are denoted by greek letters $\mu$ or $\nu$.  Further, we write $d_{W,\al}$ to denote the $1$-Wasserstein distance on $\cP_1(\cC^{0,\al}_g([0,T];\bbR^{2d}))$.  We call $B^{ij}$ the path $(B^i,B^j)$ and $\bfB^{ij}=(B^{ij},\bbB^{ij})$ the corresponding rough paths.  We define it as $\bfB=S(B)$ (this ensures we can apply the extended contraction principle on the whole space), but, as far as the law is concerned, it is equivalent to define $\bfB$ via Statonovich integral (see the section on rough paths). The enhanced empirical measure $\bfL^B_n$ is defined as
\begin{align*}
  \bfL_n^\bfB
  \;=\;
  \frac{1}{n^2} \sum^n_{i,j=1}\de_{(B^{ij}, \bbB^{ij})}.
\end{align*}
Recall the definition of $S$ given in \eqref{eq:def:S} and of $F:\cP_1(C^{0,\al}([0,T];\bbR^d))\rightarrow \cP_1(\cC^{0,\al}_g([0,T];\bbR^{2d}))$ (formula \eqref{defFk} in the case $k=2$) as the map
\begin{align}\label{def_F}
  F: Q\mapsto (Q \otimes Q) \circ S^{-1}.
\end{align}
Recall also the definition of the projection $\pi_1$ as $\pi_1(\bfX) = X^1$ for any element $\bfX = ((X^1, X^2), \bbX)$ in $\cC_g^{0,\al}([0, T]; \bbR^{2d})$.
\begin{theorem}
  \label{thm:Sanov_ebm}
  Let $\{B^i : i \in \bbN\}$ be a family of independent $d$-dimensional Brownian motion, with initial measure $\la$ and assume that there exists $c, \ve > 0$ such that
  \begin{align}\label{eq:exp:intcond}
    \int_{\bbR^d} \me^{c|x|^{1+\ve}}\, \la(\md x) \;<\; \infty.
  \end{align}
  The family $\{\mathop{\mathrm{Law}}(\bfL^\bfB_n) : n \in \bbN\}$ satisfies a LDP on $\cP_1(\cC^{0,\al}_g([0, T]; \bbR^{2d}))$ endowed with the 1-Wasserstein metric, with scale $n$ and good rate function $\bfI$ given by
  \begin{align}
    \bfI(\mu)
    \;=\;
    \begin{cases}
      H(\mu\circ\pi_1^{-1}\,|\, P^{\{d\}}), \quad
      &\text{if }\; \mu=F(\mu\circ\pi_1^{-1}),\\[1ex]
      \infty,
      &\text{otherwise}. 
    \end{cases}
  \end{align}
\end{theorem}
The basic fact, that invites us to use the extended contraction principle, is the following lemma.
\begin{lemma}\label{lemma:enh_true}
  The enhanced empirical measure $\bfL^\bfB_n$ is a.s.\ the image of the (true) empirical measure $L^B_n$ under the map $F\!: Q \mapsto (Q \otimes Q) \circ S^{-1}$.
\end{lemma}
\begin{proof}
  The image measure of $L^B_n$ under $F$ is given by
  \begin{align}
    \frac{1}{n^2} \sum^n_{i,j=1} \de_{S(B^{ij})}.
  \end{align}
  By Proposition \ref{prop:conv_S}, the Stratonovich rough paths $\bfB^{ij}$ coincides a.s.\ with $S(B^{ij})$, hence the image measure of $L^B_n$ under $F$ coincides a.s.\ with $\bfL^\bfB_n$.
\end{proof}
In order to apply the extended contraction principle, we introduce a continuous approximation $F_m$ to the map $F$, defined in this way. Given a continuous trajectory $Y$, we define its piecewise linear approximation $Y^{(m)}$ as
\begin{align*}
  Y^{(m)}(t)
  \;=\;
  Y\left(\frac{[mt]}{m}\right) \,+\,
  m \left(%
    Y\left(\frac{[mt]+1}{m}\right) \,-\, Y\left(\frac{[mt]}{m}\right)
  \right)\, \left(t-\frac{[mt]}{m}\right).
\end{align*}
The iterated integral of $Y^{(m)}$ is classically defined as Riemann integral, precisely
\begin{align*}
  \big(\bbY^{(m)}_t\big)^{ij}
  \;=\;
  \int^t_0 Y^{(m),i}_s\; \md Y^{(m),j}_s.
\end{align*}
Now we set $F_m$ as
\begin{align*}
  F_m\!: \cP_1\big(C^{0,\al}([0,T]; \bbR^d)\big) 
  \;\longrightarrow\;
  \cP_1\big(\cC^{0,\al}_g([0,T]; \bbR^{2d})\big),
  \qquad
  Q \;\longmapsto\; \big(Q \otimes Q\big) \circ (S^{(m)})^{-1}
\end{align*}
where 
\begin{align*}
  C^{0,\al}([0,T]; \bbR^{2d})
  \;\ni\;
  Y
  \;\longmapsto\;
  S^{(m)}(Y)
  \;\ldef\;
  \big(Y^{(m)}, \bbY^{(m)})\big)
  \:=\;S(Y^{(m)})
  \;\in\;
  \cC^{0,\al}_g([0,T]; \bbR^{2d}).
\end{align*}
Note that this $S^{(m)}$ is defined as $S^k$, but replacing the dyadic approximation with the approximation at step $1/m$. We denote by $\bfL^{\bfB^{(m)}}_n$ the enhanced empirical measure associated with $B^{(m)}$, namely $\bfL^{\bfB^{(m)}}_n=F_m(L^B_n)$. Notice that, for each $m$, $S^{(m)}$ is continuous with at most linear growth (this is due to the use of the homogeneous distance $d_\al$) and the map $Q \mapsto Q \otimes Q$ is continuous with respect to the $1$-Wasserstein metrics on $\cP_1(C^{0,\al}([0,T], \mathbb{R}^d))$ and $\cP_1(C^{0,\al}([0,T], \mathbb{R}^{2d}))$ (Lemma \ref{continuity_doubling} in the Appendix). So $F_m$ is continuous in the $1$-Wasserstein metric (by Corollary \ref{cont_lingrowth} in the Appendix).

In the proceeding lemmata, we show that the approximation given by $F_m$ is indeed exponentially good, in the sense of the extended contraction principle (as in \cite[Lemma~2.1.4]{DeSt89}).  The main tool is the following lemma, which follows from \cite{FrVi10} (see Corollary 13.21 and Exercise 13.22, a proof is given in the Appendix), which gives an exponential bound for the approximation.
\begin{lemma}\label{lemma:Strat_approx}
  Let $\bfB$ the Stratonovich enhanced Brownian motion on $\bbR^e$, let $\bfB^{(m)}$ be its piecewise linear approximation, defined as before. Fix $\al < 1/2$. Then, for every $\eta$ in $(0, 1/2-\al)$, there exists $c > 0$ such that
  \begin{align}
    \sup_{m \geq 1}
    \mean\!\Big[\exp\Big(c m^{\eta/2}\, d_{\al}\big(\bfB, \bfB^{(m)}\big)\Big)\Big]
    \;<\;
    \infty
  \end{align}
\end{lemma}
As a first step, we establish the exponential tightness of the approximation $\bfL_n^{\bfB^{(m)}}$ of $\bfL_m^{\bfB}$.
\begin{lemma}\label{lemma:LDP_approx_1}
  For any $\de > 0$, it holds
  \begin{align}
    \lim_{m \to \infty} \limsup_{n\to \infty}\,
    \frac{1}{n}\, \log
    \prob\!\Big[d_{W, \al}\big(\bfL^{\bfB}_n,\bfL^{\bfB^{(m)}}_n\big) > \de \Big] 
    \;=\;
    -\infty.
  \end{align}
\end{lemma}
\begin{proof}
  Consider the coupling measure $\frac{1}{n^2}\sum^n_{i,j=1} \de_{(\bbB^{ij},\bbB^{(m),ij})}$ with marginals $\bfL^{\bfB^{(m)}}_n$ and $\bfL^{\bfB}_n$.  Then, in view of \eqref{eq:def:Wasserstein}, we obtain that
  \begin{align}\label{eq:Wasserstein:ub}
    d_{W,\al}\big(\bfL^{\bfB}_n,\bfL^{\bfB^{(m)}}_n\big)
    \;\leq\;
    \frac{1}{n^2} \sum^n_{i,j=1}\, d_{\al}(\bfB^{ij}, \bfB^{(m),ij}),
  \end{align}
  where we used the fact that the map $(\bfX,\bfX')\mapsto d_\al(\bfX,\bfX')$ is Lipschitz continuous.  By means of Hoeffding's decomposition \cite{Ho63}, the right-hand side of \eqref{eq:Wasserstein:ub} can be rewritten as
  \begin{align*}
    \frac{1}{n(n-1)}\, \sum_{\substack{i,j = 1\\i\,\ne\,j}}^n H_m(i,j)
    \;=\;
    \frac{1}{n!}\, \sum_{\si \in \cS_n}\,
    \frac{1}{\lfloor n/2 \rfloor}\, \sum_{i=1}^{\lfloor n/2 \rfloor}
    H_m\big(\si(2i-1),\si(2i)\big),
  \end{align*}
  where $\cS_n$ denotes the set of all permutations of $\{1, \ldots, n\}$ and
  \begin{align*}
    H_{m,n}(i, j)
    \;\equiv\;
    H_m(i,j)
    \;\ldef\;
    \frac{n-1}{n}\, 
    d_{\al}\,(\bfB^{ij}, \bfB^{(m),ij})
    \,+\, \frac{1}{n}\, d_{\al}(\bfB^{ii}, \bfB^{(m),ii}).
  \end{align*}
  Hence, an application of the Markov inequality and Jensen's inequality gives, for any $C > 0$ and any $n$ and $m$,
  \begin{align}
    \prob\! \Big[d_{W,\al}\big(\bfL_n^{\bfB},\bfL_n^{\bfB^{(m)}}\big) > \de \Big]
    &\;\leq\;
    \prob\!%
    \bigg[
      \frac{1}{n!} \sum_{\sigma\in\cS_n} \frac{1}{\lfloor n/2\rfloor}
      \sum^{\lfloor n/2\rfloor}_{i=1} H_m\big(\si(2i-1),\si(2i)\big) > \de
    \bigg]
    \nonumber\\[.5ex]
    &\;\leq\;
    \me^{-C\de}\,
    \mean\!%
    \bigg[
      \exp\!%
      \bigg(
        \frac{C}{n!}\, \sum_{\si \in \cS_n} \frac{1}{\lfloor n/2\rfloor}
        \sum_{i=1}^{\lfloor n/2\rfloor} H_m\big(\si(2i-1),\si(2i)\big)
      \bigg)
    \bigg]
    \nonumber\\[.5ex]
    &\;\leq\;
    \me^{-C \de}\,
    \frac{1}{n!} \sum_{\si \in \cS_n}
    \mean\!%
    \bigg[
      \exp\!%
      \bigg(
      \frac{C}{\lfloor n/2\rfloor}
      \sum^{\lfloor n/2\rfloor}_{i=1} H_m\big(\si(2i-1),\si(2i)\big)
      \bigg)
    \bigg].
    \nonumber
  \end{align}
  Here, we see the advantage of Hoeffding's decomposition: by using the mutually independence of $\{H\big(\si(2i-1), \si(2i)\big) : i = 1, \ldots, \lfloor n/2\rfloor\}$ we finally get that
  \begin{align}\label{eq:Wasserstein:exp:ub}
    \prob\!\Big[d_{W,\al}\big(\bfL_n^{\bfB}, \bfL_n^{\bfB^{(m)}}\big) > \de \Big]
    \;\leq\;
    \me^{-C \de}\,
    \mean\!%
    \bigg[
      \exp\!\bigg(\frac{C}{\lfloor n/2\rfloor}\, H_m(1,2)\bigg)
    \bigg]^{\lfloor n/2\rfloor}.
  \end{align}
  On the other hand, by choosing $C = c m^{\eta} n / (6(c'\vee 1))$ for some $c' < \infty$ such that $d_\al(\bfB^{11},\bfB^{(m),11}) \leq c'\, d_\al(\bfB^1,\bfB^{(m),1})$, Lemma~\ref{lemma:Strat_approx} implies that, for any $\eta \in (0, 1/2 - \al)$ and any $n \geq 2$,
  \begin{align*}
    &\sup_{m \geq 1}
    \mean\!%
    \bigg[
      \exp\!%
      \bigg(
        \frac{c}{6(c'\vee 1)}\,\frac{m^{\eta} n}{\lfloor n/2\rfloor}\, H_m(1,2)
      \bigg)
    \bigg]
    \\[1ex]
    &\mspace{36mu}\leq\;
    \sup_{m \geq 1}
    \mean\!%
    \Big[
      \exp\!\Big(cm^{\eta}\, d_\al(\bfB^{12},\bfB^{(m),12}) \Big)
    \Big]^{1/2}
    \mean\!%
    \Big[
      \exp\!\Big(cm^{\eta}\, d_\al(\bfB^1,\bfB^{(m),1}) \Big)
    \Big]^{1/2}
    \;<\;
    \infty,
  \end{align*}
  By combining this estimate with \eqref{eq:Wasserstein:exp:ub}, the assertion follows.
\end{proof}
\begin{lemma}\label{lemma:LDP_approx_2}
  For every $a < \infty$, it holds
  \begin{align}
    \lim_{m \to \infty}\, \sup_{Q: H(Q\mid P^{\{d\}}) \leq a}\, 
    d_{W,\al}\big(F_m(Q), F(Q)\big)
    \;=\;
    0.
  \end{align}
\end{lemma}
\begin{proof}
  Using the coupling $(Q\otimes Q)\circ (S^{(m)}(X),S(X))^{-1}$, we get
  \begin{align*}
    d_{W,\al}\big(F_m(Q), F(Q)\big)
    &\;\leq\;
    \int_{C^{0,\al}([0,T], \bbR^{2d})} d_\al\big(S^{(m)}(X), S(X)\big)\,
    Q\otimes Q(\md X)
    \\[.5ex]
    &\;=\;
    \int_{C^{0,\al}([0,T],\bbR^{2d})}
      d_\al\big(S^{(m)}(X), S(X)\big)\,
      \bigg(\frac{\md Q}{\md P^{\{d\}}}\otimes\frac{\md Q}{\md P^{\{d\}}}\bigg)(X)\,
    \big(P^{\{d\}} \otimes P^{\{d\}}\big)(\md X).
  \end{align*}
  The idea is the following: For any $Q$ with bounded entropy, $\frac{\md Q}{\md P^{\{d\}}}\otimes\frac{\md Q}{\md P^{\{d\}}}$ has a uniform $L\log L$ bound with respect to the Wiener measure $P^{\{d\}}$.  Hence, the lemma is proven if the norm of $d_\al(S^{(m)}(X),S(X))$ in the dual space of $L\log L$, again with respect to $P^{\{d\}}$, converges to $0$. This convergence follows by an exponential control of $d_\al(S^{(m)}(X),f(X))$ under $P^{\{d\}}$, which is a consequence of Lemma \ref{lemma:Strat_approx}.

  To make this argument work, we use the theory of Orlicz space. Let $\Phi$, $\Psi\!: [0, \infty) \to [0, \infty)$ be a complementary Young pair of $N$-functions defined by
  \begin{align*}
    \Phi(r)
    \;=\;
    \frac{1}{2}\, r^2\, 1_{r \leq 1} \,+\,
    \Big(\me^{r-1} -\frac{1}{2}\Big)\, 1_{r > 1},
    \qquad \text{and} \qquad
    \Psi(r)
    \;=\;
    \frac{1}{2}\, r^2\, 1_{r\le 1} \,+\,
    \Big( r \log r + \frac{1}{2} \Big)\, 1_{r > 1}.
  \end{align*}
  Further, on a given measure space $(\La, \Si, \mu)$, introduce for any $g, h\!: \La \to [0, \infty)$ measurable
  \begin{align*}
    \|g\|_{L_{\exp}}
    \;\ldef\;
    \inf_{k>0}
    \left\{
      \frac{1}{k}\, \Big(1 + \int_{\La}\Phi(kg)\, \md \mu \Big)
    \right\}
    \qquad \text{and} \qquad
    \|h\|_{L\log L}
    \;\ldef\;
    \inf_{k>0}
    \left\{
      \frac{1}{k}\Big(1+\int_{\La} \Psi(kh)\, \md \mu\Big)
    \right\}.
  \end{align*}
  Then, the classical Orlicz-Birnbaum estimate, see \cite[Section~3.3]{RaRe91}m implies that for any measurable, nonnegative functions $g$ and $h$, it holds
  \begin{align}\label{eq:Holder_Orlicz}
    \int_{\La} g h\, \md\mu
    \;\leq\;
    4\, \|g\|_{L\log L}\, \|h\|_{L_{\exp}}.
  \end{align}
  In particular, by using the explicit form of the Orlicz pair $(\Phi, \Psi)$ the follwing estimates holds for any measurable, nonnegative functions $g, h$ and $k > 0$
  \begin{align}\label{eq:Holder_Orlicz_2}
    \int_{\La} g h\, \md\mu
    \;\leq\;
    \frac{4}{k}\,
    \Big(1 + \int_{\La} \exp(kg)\, \md\mu\Big)\,
    \Big(2 + \int_{\La} h \log h\,, \md\mu\Big).
  \end{align}
  By applying \eqref{eq:Holder_Orlicz_2} with $\La = C^{0,\al}(\bbR^{2d})$, $\mu = P^{\{d\}}\otimes P^{\{d\}}$, $g = d_{\al}(S^{(m)}(X),S(X))$, $h = \frac{\md Q}{\md P^{\{d\}}}\otimes\frac{\md Q}{\md P^{\{d\}}}$ we get
  \begin{align}
    d_{W,\al}\big(F_m(Q), F(Q)\big)
    \;\leq\;
    \frac{4}{k}(2+2a)
    \bigg(
      1 \,+\,
      \mean\!\Big[\exp\!\big(k\, d_{\al}\big(\bfB^{12}, \bfB^{12,(m)}\big)\Big]
    \bigg),
  \end{align}
  where we used that $\int_\Lambda h \log h\, \md\mu = 2H(Q|P^{\{d\}})\leq 2a$.  Finally, by choosing $k = c m^{\eta/2}$, a further application of Lemma~\ref{lemma:Strat_approx} yields
  \begin{align*}
    d_{W,\al}\big(F_m(Q), F(Q)\big) \;\leq\; \frac{4}{cm^{\eta/2}}(2+2a)(1+C),
  \end{align*}
  which completes the proof.
\end{proof}
\begin{proof}[Proof of Theorem \ref{thm:Sanov_ebm}]
  By Sanov theorem~\ref{SanovWasserstein} and Remark~\ref{SanovWiener}, the extended contraction principle (\cite{DeSt89}, Lemma 2.1.4) together with the Lemma \ref{lemma:LDP_approx_1} and \ref{lemma:LDP_approx_2} show that $\{\mathop{\mathrm{Law}}(\bfL^{\bfB^{(m)}}_n): n \in \bbN\}$ satisfies an LDP with scale $n$ and good rate function given by
  \begin{align*}
    \mu\mapsto \inf\Big\{
      H(Q \,|\, P) \mspace{12mu}\Big|\mspace{12mu}
      Q \in \cP_1\big(C^{0,\al}([0,T],\bbR^d)\big)
      \; \text{ and } \;
      F(Q) = \mu
    \Big\}.
  \end{align*}
  It is easy to see that this rate function coincides with the $\bfI$ defined in Theorem \ref{thm:Sanov_ebm}.
\end{proof}
We close the section with the convergence (in probability) of the enhanced empirical measures, which follows from the LDP (as well known in large deviations theory).
\begin{corro}
  The sequence of $\cP_1(\cC^{0,\al}_g([0,T];\bbR^{2d}))$-valued random variables $\{\bfL^\bfB_n : n \in \bbN\}$ converges in probability (and in law) to the constant random variable $\bfP^{\{2d\}}$, the enhancement of the $2d$-Wiener measure, that is the law on $\cC^{0,\al}_g([0,T];\bbR^{2d})$ of $(B^{12},\bbB^{12})$.
\end{corro}
\begin{proof}
   The result is a consequence of the LDP for the laws of $\bfL^\bfB_n$ and of the fact that the good rate function has a unique zero in $\bfP^{\{2d\}}$.
\end{proof}

\section{The modified Wasserstein space}    \label{section_modW}
As already mentioned in the Introduction, in view of our application (Theorem \ref{mainres}), we will have to deal with maps of the form
\begin{align}
  \mu
  \;\longmapsto\;
  \int_{\cC^{0,\al}_g} \int_0^T f(X)\, \md\bfX\, \mu(\md\bfX)
\end{align}
and we would like these maps to be continuous (to apply standard tools of large deviations theory). On one side, we know that a map $\mu \mapsto \int G\, \md\mu$ is continuous in the $1$-Wasserstein metric if $G$ is continuous with at most linear growth. But on the other side, by Theorem~\ref{thm:cont_rough_int}, the rough path integral has a growth of order at most $1/\al$, in particular a more than linear growth (with respect to the homogeneous rough paths norm).\footnote{The path-by-path estimate in Theorem~\ref{thm:cont_rough_int} is optimal.}  This creates a problem.  Following \cite{CLL13}, we introduce a new function $N$ of $\bfX$ with good concentration properties (w.r.t. to Brownian rough paths) such that the rough integral has at most linear growth with respect to $N$. We then device a strengthened topology, on a restriction of the space $\cP_1(\cC^{0,\al}_g([0,T];\bbR^e))$, which allows us to use as test functions also functions with linear growth with respect to such $N$.

In this new topology we prove the large deviation principle for the enhanced empirical measures, as a consequence of the LDP in the $1$-Wasserstein metric, via inverse contraction principle. This amounts to verify exponential tightness in the new topology, which can be proved using again Hoeffding decomposition and also Gaussian estimates for Brownian rough paths.
\begin{remark}
  One may ask why we do not take simply $N(\bfX) = \|X\|_\al^{1/\al}$, or allow for $p$-Wasserstein distance, for $p=1/\al$. The reason is that, with this choice of $N$, we are not able to prove a Sanov-type theorem for the enhanced empirical measure. Actually, in \cite{WWW10}, it is proved that a large deviation result in the $p$-Wasserstein distance does not hold for any $p>2$ (and actually also for $p=2$), as a consequence of the lack of exponential integrability of $\|X\|_\al^p$.
\end{remark}

\subsection{A modified Wasserstein topology}\label{mod_W_subsection}
For the definition of $N$, consider the following sequence of stopping times: given $\bfX$ in $\cC^{0,\al}_g([0,T];\bbR^e)$, we define
\begin{align}
  \tau^\al_0(\bfX) \;=\;0,
  \qquad
  \tau_{i+1}^{\al}(\bfX)
  \;=\;
  \inf\big\{%
    t > \tau_i^{\al}(\bfX)
    \,:\,
    \|\bfX\|_{(1/\al)-\mathrm{var},[\tau^\al_i(\bfX),t]} \geq 1
  \big\},
  \quad i \in \bbN.
\end{align}
Here $\|X\|_{(1/\al)-\mathrm{var},[s,t]}$ is the $(1/\al)$-variation of $\bfX$, as group-valued path, in the interval $[s, t]$, see \cite{FrVi10} for precise definition. What we need here is that the norm $\|X\|_{(1/\al)-\mathrm{var},[s, t]}$ is a continuous function of $\bfX$, in the space $\cC^{0,\al}_g([0,T];\bbR^e)$, for fixed $s$, $t$, and it is independent of the initial datum $X_0$.  Notice, that it is also a continuous function of $s$, $t$, for fixed $\bfX$, and it is monotone in $s$ and $t$, in the sense that $\|X\|_{(1/\al)-\mathrm{var},[s',t']} \leq \|X\|_{(1/\al)-\mathrm{var},[s,t]}$ for any $s \leq s' < t' \leq t$.  We define $N = N_{\al}$ as
\begin{align} \label{def:N}
  N_\al(\bfX) \;\ldef\; \sup\big\{i \in \bbN \,:\, \tau_i^{\al}(\bfX) < T \big\}.
\end{align}
We omit $\al$ when not necessary. The following lower-semicontinuity property of $N$ will be proved in the Appendix.
\begin{lemma}\label{Nlsc}
  The function $N$ is lower semi-continuous on $\cC^{0,\al}_g([0,T];\bbR^e)$.
\end{lemma}
The next lemma gives the desired sublinear growth of the rough integral in terms of $N$, see \cite{FrHa14} for a proof.
\begin{lemma}
  Let $f$ be a function in $C^2_b(\bbR^e)$ and let $\bfX$ be in $\cC^{0,\al}_g([0,T];\bbR^e)$. Then it holds, for some constant $C_f$ depending on $f$,
  \begin{align}\label{eq:lineargrowthRP}
    \left| \int^T_0 f(X)\, \md\bfX \right|
    \;\leq\;
    C_f\, \big(1 + N(\bfX) \big).
  \end{align}
\end{lemma}
Now we introduce a modified topology, on a restriction of $\cP_1(\cC^{0,\al}_g([0,T];\bbR^e))$, in order to deal with functionals of the form $\mu \mapsto \int G\, \md\mu$ for some continuous $G$ with $G(\bfX) \leq C(1+N(\bfX))$.

First, for given $\ve > 0$, we introduce the space
\begin{align}
  &\cP_{(\|\cdot\|+N)^{1+\ve}}\big(\cC^{0,\al}_g([0,T];\bbR^e)\big)
  \nonumber\\[1ex]
  &\mspace{36mu}\ldef\;
  \Big\{
    \mu \in \cP_1(\cC^{0,\al}_g([0,T];\bbR^e))
    :
    \int_{\cC^{0,\al}_g}
      (|X_0| + \|\bfX\|_{\al} + N_{\al}(\bfX))^{1+\ve}\,
    \mu(\md\bfX)
    <
    \infty
  \Big\}
\end{align}
of probability measures with finite $(|X_0|+\|\cdot\|_\al+N_\al)^{1+\ve}$.
\begin{definition}
  Let $\mu_n$, $n \in \bbN$, $\mu$ be in $\cP_{(\|\cdot\|+N)^{1+\ve}}\big(\cC^{0,\al}_g([0,T];\bbR^e)\big)$. We say that $\{\mu_n : n \in \bbN\}$ converges to $\mu$ in the $(\|\cdot\|+N)^{1+\ve}$-Wasserstein topology if the following two conditions hold:
  \begin{enumerate}
  \item $\{\mu_n : n\in \bbN\}$ converges to $\mu$ in the weak topology, i.e.\ wrt. any test function in $C_b(\cC^{0,\al}_g([0,T];\bbR^e))$;
  \item we have
    \begin{align}
      \sup_{n \in \bbN}\,
      \int_{\cC^{0,\al}_g} (|X_0|+\|\bfX\|_\al+N_\al(\bfX))^{1+\ve}\, \mu_n(\md\bfX)
      \;<\;
      \infty.
    \end{align}
  \end{enumerate}
  We say that a subset $C$ of $\cP_{(\|\cdot\|+N)^{1+\ve}}(\cC^{0,\al}_g([0,T];\bbR^e))$ is closed in the $(\|\cdot\|+N)^{1+\ve}$-Wasserstein topology if it is closed under convergence of sequences.
\end{definition}
Due to \eqref{eq:equiv_dist_RP}, $|X_0|+\|\bfX\|_\al$ is equivalent to $d_{\al}(\bfX,0)$.  Thus, every sequence converging in the $(\|\cdot\|+N)^{1+\ve}$-Wasserstein topology converges also in the $1$-Wasserstein metric (remind the characterization of $1$-Wasserstein distance in Lemma \ref{char_1Wasserstein}). Hence the $(\|\cdot\|+N)^{1+\ve}$-Wasserstein topology is stronger than the $1$-Wasserstein topology. Note that the $(\|\cdot\|+N)^{1+\ve}$-Wasserstein topology might not be metrizable. However, the following result, proved in the Appendix, gives the properties needed for large deviations analysis.
\begin{lemma}\label{Mod_regular}
  The space $\cP_{(\|\cdot\|+N)^{1+\ve}}(\cC^{0,\al}_g([0,T];\bbR^e))$ , with the $(\|\cdot\|+N)^{1+\ve}$-Wasserstein topology, is a regular Haussdorff space.
\end{lemma}
Furthermore, since $N$ is lower semi-continuous, by Corollary \ref{lsc_functional} the functional
\begin{align*}
  \mu
  \;\longmapsto\;
  \int_{\cC^{0,\al}_g} (|X_0|+\|\bfX\|_\al+N_\al(\bfX))^{1+\ve}\, \mu(\md\bfX)
\end{align*}
is sequentially lower semi-continuous with respect to the $1$-Wasserstein topology: if $\{\mu_n : n \in \bbN\}$ converges to $\mu$ in the $1$-Wasserstein or in the $(\|\cdot\|+N)^{1+\ve}$-Wasserstein topology, then
\begin{align}\label{eq:lsc}
  \int_{\cC^{0,\al}_g} (|X_0| + \|\bfX\|_\al + N_\al(\bfX))^{1+\ve}\, \mu(\md\bfX)
  \;\leq\;
  \liminf_n\,
  \int_{\cC^{0,\al}_g} (|X_0| + \|\bfX\|_\al + N_\al(\bfX))^{1+\ve}\, \mu_n(\md\bfX).
\end{align}
\begin{prop}\label{Nlingrowth}
  Assume that $\{\mu_n : n \in \bbN\}$ converges to $\mu$ in the $(\|\cdot\|+N)^{1+\ve}$-Wasserstein topology. Let $G$ be a continuous function on $\cC^{0,\al}_g([0,T];\bbR^e)$ such that $G(\bfX)\le C(1+|X_0|+\|\bfX\|_\al+N_\al(\bfX))$, as for example the rough integral.  Then,
  \begin{align}
    \lim_{n \to \infty}\, \int_{\cC^{0,\al}_g} G(\bfX)\, \mu_n(\md\bfX)
    \;=\;
    \int_{\cC^{0,\al}_g} G(\bfX)\,\mu(\md\bfX).
  \end{align}
\end{prop}
\begin{proof}
  For any $m$ positive integer, let $G_m$ be the continuous bounded function defined from $G$ with truncation at level $m$, that is $G_m= G 1_{|G| \leq m} + m 1_{G>m} - m 1_{G<-m}$.  We have for every $m$, $n$,
  \begin{align*}
    \bigg| \int_{\cC^{0,\al}_g} G\, \md(\mu_n-\mu) \bigg|
    \;\leq\;
    \bigg| \int_{\cC^{0,\al}_g} G_m\, \md(\mu_n-\mu) \bigg|
    \,+\,
    \int_{\cC^{0,\al}_g} \big| G_m - G \big|\, \md(\mu_n+\mu)
  \end{align*}
  Notice that the condition $G(\bfX) \leq C(1+|X_0|+\|\bfX\|+N(\bfX))$ implies (for $m$ with $m/C > 4$) that
  \begin{align*}
    \big| G(\bfX) - G_m(\bfX) \big|
    \;\leq\;
    2C (|X_0| + \|\bfX\| + N(\bfX))\,1_{\|\bfX\|+N(\bfX)>m/(2C)}.
  \end{align*}
  So it holds for any $n$
  \begin{align*}
    \int_{\cC^{0,\al}_g} \big| G_m - G \big|\, \md\mu_n
    &\;\leq\;
    2C\,
    \int_{\cC^{0,\al}_g}
      (|X_0|+\|\bfX\|+N(\bfX))\, 1_{\|\bfX\|+N(\bfX)>m/(2C)}\,
    \md\mu_n
    \\[.5ex]
    &\;\le\;
    m^{-\ve} (2C)^{1+\ve}\,
    \int_{\cC^{0,\al}_g} (|X_0|+\|\bfX\|+N(\bfX))^{1+\ve}\, \md\mu_n
    \;\leq\;
    m^{-\ve} (2C)^{1+\ve} D,
  \end{align*}
  where $D \ldef \sup_n\int_{\cC^{0,\al}_g}(|X_0|+\|\bfX\|+N(\bfX))^{1+\ve}\, \mu_n(\md\bfX)$ is bounded by assumption.  The same estimates holds also for $\mu$ in place of $\mu_n$, by the lower semi-continuity property \eqref{eq:lsc}. Hence, for any $\rho>0$, we can find $m_\rho$ such that
  \begin{align*}
    \int_{\cC^{0,\al}_g} \big| G_{m_\rho} - G \big|\, \md(\mu_n+\mu) \;<\; \rho.
  \end{align*}
  Fix such $m_{\rho}$.  Since $G_{m_\rho}$ is continuous bounded, there exists $n_{\rho} < \infty$ such that, for every $n \geq n_{\rho}$,
  \begin{align*}
    \bigg| \int_{\cC^{0,\al}_g} G_{m_\rho}\, \md(\mu_n-\mu) \bigg| \;<\; \rho.
  \end{align*}
  So we conclude that, for every $n \geq n_{\rho}$,
  \begin{align*}
    \bigg| \int_{\cC^{0,\al}_g} G\, \md(\mu_n-\mu) \bigg| \;<\; 2\rho.
  \end{align*}
  The proof is complete.
\end{proof}
We conclude this subsection with a lemma on compact sets on this space. This will be useful in view of exponential tightness on $\cP_{(\|\cdot\|+N)^{1+\ve}}(\cC^{0,\al}_g([0,T];\bbR^e))$. Recall that, given a topology $\tau$ (i.e.\ the set of all open sets), its restriction $\tau_A$ to a set $A$ is given by $\{B \cap A : B \in \tau\}$.
\begin{lemma}\label{char_bdd_cpt}
  \begin{enumerate}
  \item For any $R > 0$, the $(\|\cdot\|+N)^{1+\ve}$-Wasserstein topology restricted on the set
    \begin{align}
      \bar{B}(R)
      \;=\;
      \Big\{
        \mu \in \cP_{(\|\cdot\|+N)^{1+\ve}}\big(\cC_g^{0, \al}([0,T];\bbR^e)\big)
        :
        \int_{\cC^{0,\al}_g} (|X_0|+\|\cdot\|+N)^{1+\ve}\, \md\mu \leq R
      \Big\}
    \end{align}
    coincides with the $1$-Wasserstein topology restricted there.
  \item Let $H$ be a subset of $\cP_{(\|\cdot\|+N)^{1+\ve}}(\cC_g^{0, \al}([0,T];\bbR^e))$, which is compact in the $1$-Wasserstein metric and is ``bounded'' in $\cP_{(\|\cdot\|+N)^{1+\ve}}(\cC_g^{0, \al}([0,T];\bbR^e))$, in the sense that
    \begin{align}\label{eq:bdd_mod_W}
      \sup_H\,
      \int_{\cC^{0,\al}_g} (|X_0|+\|\bfX\|+N(\bfX))^{1+\ve}\, \mu(\md\bfX)
      \;<\;
      \infty.
    \end{align}
    Then, $H$ is compact in $\cP_{(\|\cdot\|+N)^{1+\ve}}(\cC_g^{0, \al}([0,T];\bbR^e))$ (with its $(\|\cdot\|+N)^{1+\ve}$-Wasserstein topology).
  \end{enumerate}
\end{lemma}
\begin{proof}
  For the first part, every closed set in $\bar{B}(R)$ with respect to the (restricted) $1$-Wasserstein topology is also closed with respect to the (restricted) $(\|\cdot\|+N)^{1+\ve}$-Wasserstein topology, this being stronger. Conversely, let $C$ be a closed subset of $\bar{B}(R)$ in the restricted $(\|\cdot\|+N)^{1+\ve}$-Wasserstein topology; notice that $C$ is closed also in the (not restricted) $(\|\cdot\|+N)^{1+\ve}$-Wasserstein topology, since $\bar{B}(R)$ is closed in this topology (by the lower semicontinuity property \eqref{eq:lsc}). Let $(\mu_n)_n$ be a sequence in $C$, converging to $\mu$ in the $1$-Wasserstein metric. Since $C$ is in $\bar{B}(R)$, the uniform bound
  \begin{align}
    \sup_{n \in \bbN}\,
    \int_{\cC^{0,\al}_g} (|X_0|+\|\bfX\|+N(\bfX))^{1+\ve}\, \mu_n(\md\bfX)
    \;<\;
    \infty
  \end{align}
  holds, hence $\mu_n$ converges to $\mu$ also in the $(\|\cdot\|+N)^{1+\ve}$-Wasserstein topology.  Furthermore $\mu$ is also in $\bar{B}(R)$, by the lower semicontinuity property \eqref{eq:lsc}. Since $C$ is closed in the $(\|\cdot\|+N)^{1+\ve}$-Wasserstein topology, $\mu$ must be in $C$ and so $C$ is closed also in the $1$-Wasserstein topology. The first statement is proved.

  The second part follows from the first one (as a general fact in topology), we give a proof for completeness.  Let $(A_i)_{i\in I}$ be a family of open sets, in the $(\|\cdot\|+N)^{1+\ve}$-Wasserstein topology, whose union contains $H$ and take $R>0$ such that $H$ is contained in $\bar{B}(R)$. Consider $\tilde{A}_i:=A_i\cap\bar{B}(R)$, which are open sets in the $(\|\cdot\|+N)^{1+\ve}$-Wasserstein topology restricted on $\bar{B}(R)$.  By the first statement, they are open also in the restricted $1$-Wasserstein topology on $\bar{B}(R)$.  That is, there exist $B_i$ (subsets of $\cP_1(\cC^{0,\al}_g)$), open sets in the $1$-Wasserstein topology, such that $\tilde{A}_i=B_i\cap \bar{B}(R)$. Actually, since $\bar{B}(R)$ is closed in every topology under consideration, one can choose $B_i=A_i\cup \cP_1(\cC^{0,\al}_g)\setminus \bar{B}(R)$.  In particular $(B_i)_{i\in I}$ is a family of open sets, in the $1$-Wasserstein metric, covering $H$. By the compactness of $H$ in the $1$-Wasserstein metric, we can extract a finite subset $\{i_1, \ldots, i_m\}$ of $I$ such that $\bigcup_{1 \leq k \leq m} B_{i_k}$ contains $H$. Since $\tilde{A}_i=B_i\cap \bar{B}(R)$ and $H$ is in $\bar{B}(R)$, also $\bigcup_{1 \leq k \leq m} A_{i_k}$ contains $H$. The proof is complete.
\end{proof}

\subsection{The LDP in the modified Wasserstein space}
In this section we prove the LDP for the enhanced empirical space, in the stronger $(\|\cdot\|+N)^{1+\ve}$-Wasserstein topology, again for the double layer case ($k=2$ will be fixed and often omitted in the notation). Recall the definition of $F$ and $S$ in \eqref{def_F}, \eqref{eq:def:S}. Recall that the Brownian motions $B^i$ (and their corresponding rough paths) start from measure $\lambda$ satisfying \eqref{eq:exp:intcond}. Here we assume $\ve$ to be the one appearing in condition \eqref{eq:exp:intcond}.  Mind that we need large deviation tools on a regular Haussdorff spaces (as in \cite{DeZe10}) and not just on metric spaces.
\begin{theorem}\label{thm:Sanov_strong_W}
  The sequence $\mathop{\mathrm{Law}}(\bfL^\bfB_n)_n$ satisfies a LDP on $\cP_{(\|\cdot\|+N)^{1+\ve}}(\cC^{0,\al}_g([0,T];\bbR^{2d}))$ (endowed with the $(\|\cdot\|+N)^{1+\ve}$-Wasserstein topology) with scale $n$ and good rate function
  \begin{align}
    \bfI(\mu)
    \;=\;
    \begin{cases}
      H(\mu\circ\pi_1^{-1}\,|\, P^{\{d\}}), \quad
      &\text{if }\; \mu = F(\mu \circ \pi_1^{-1}),
      \\[1ex]
      \infty,
      &\text{otherwise}. 
    \end{cases}
  \end{align}
\end{theorem}
Recall that the strategy is to use the inverse contraction principle starting from the previous Theorem \ref{thm:Sanov_ebm} and that, for this, we need the exponential tightness of the family $(\mathop{\mathrm{Law}}(\bfL^\bfB_n))_n$ in the $(\|\cdot\|+N)^{1+\ve}$-Wasserstein topology.

The main tool is the following lemma, which follows for example from \cite[Theorems~11.9 and 11.13]{FrHa14}, see also \cite[Theorem~6.3]{CLL13} adapted to our case in the Appendix, and gives an exponential bound for $N(\bfB)$.
\begin{lemma}\label{N_Gaussian}
  Let $\bfB$ the Stratonovich enhanced Brownian motion on $\bbR^e$, with initial measure $\tilde{\la}$ satisfying condition \eqref{eq:exp:intcond}. Then, for any $\al < 1/2$, $\be < 1/2$ the random variables $\|\bfB\|_{\be}$ and $N_{\al}(\bfB)$ have Gaussian tails, in particular, for some $c > 0$,
  \begin{align}
    \mean\!\Big[
      \exp\!\big(c(|B_0| + \|\bfB\|_{\be} + N_\al(\bfB))^{1+\ve}\big)
    \Big]
    \;<\;
    \infty.
  \end{align}
  The same result holds for $\bfB^{11}$ on $\bbR^{2d}$ (where $B^1$ is a Brownian motion on $\bbR^d$ starting from $\lambda$ and $\lambda$ satisfies \eqref{eq:exp:intcond}).
\end{lemma}
Here is the exponential tightness result:
\begin{lemma}\label{lemma:exp_tight_mod_W}
  The sequence $\{\mathop{\mathrm{Law}}(\bfL^\bfB_n) : n \in \bbN \}$ is exponentially tight on $\cP_{(\|\cdot\|+N)^{1+\ve}}(\cC^{0,\al}_g([0,T];\bbR^{2d}))$ with respect to the $(\|\cdot\|+N)^{1+\ve}$-Wasserstein topology.
\end{lemma}
\begin{proof}
  For any $M > 0$, we have to find a set $K_M$, compact in the $(\|\cdot\|+N)^{1+\ve}$-Wasserstein topology, such that
  \begin{align}\label{eq:exp_tight_modW}
    \limsup_n\frac{1}{n}\log \bbP\{\bfL^\bfB_n\in K_M^c\} <-M.
  \end{align}
  Our candidate for $K_M$ is
  \begin{align*}
    K_M
    \;=\;
    \Big\{
      \mu \in\cP_{(\|\cdot\|+N)^{1+\ve}} : \int_{\cC^{0,\al}_g} G\, \md\mu \leq M
    \Big\},
  \end{align*}
  where $G\!: \cC^{0,\al}_g([0,T];\bbR^{2d}) \to [0,+\infty]$ is defined by $G(\bfX) \ldef (|X_0|+\|\bfX\|_\beta+N_\al(\bfX))^{1+\ve}$ for some fixed $\be$ with $\al < \be < 1/2$.

  The compactness of $K_M$ follows from Lemma~\ref{char_bdd_cpt}, since $K_M$ satisfies the hypotheses of that result. Indeed 1) $K_M$ is ``bounded'' in $\cP_{(\|\cdot\|+N)^{1+\ve}}$ (in the sense of \eqref{eq:bdd_mod_W}), since $\|\bfX\|_{\al} \leq C\|\bfX\|_{\be}$ for some $C > 0$; 2) it is also compact in the $1$-Wasserstein metric, as a consequence of Lemma \ref{criterion_compactness} (since $G$ has more than linear growth, is lower semi-continuous and has pre-compact sublevel sets for the compact inclusion of $\cC^{\be}_g$ in $\cC^{0,\al}_g$).

  Now we verify \eqref{eq:exp_tight_modW}. We use a strategy similar to that for Lemma \ref{lemma:LDP_approx_1}. By Markov inequality, we have for any $C > 0$,
  \begin{align*}
    \prob\!\big[ \bfL^\bfB_n\in K_M^c \big]
    &\;=\;
    \prob\!\bigg[ \int_{\cC^{0,\al}_g} G\, \md\bfL^\bfB_n > M \bigg]
    \\[.5ex]
    &\;\leq\;
    \me^{-C M}\,
    \mean\bigg[
      \exp\!\bigg(C \int_{\cC^{0,\la}_g} G\, \md\bfL^\bfB_n \bigg)
    \bigg]
    \;=\;
    \me^{-CM}\,
    \mean\!\bigg[
      \exp\!\bigg( \frac{C}{n^2} \sum^n_{i,j=1} G(\bfB^{ij}) \bigg)
    \bigg].
  \end{align*}
  Exploiting Hoeffding decomposition as in the proof of Lemma \ref{lemma:LDP_approx_1}, we get
  \begin{align*}
    \prob\!\big[ \bfL^\bfB_n \in K_M^c\big]
    \;\leq\;
    \me^{-CM}\,
    \mean\!\bigg[
      \exp\!\bigg( \frac{C}{[n/2]}\, H'(1,2) \bigg)
    \bigg]^{[n/2]},
  \end{align*}
  where now
  \begin{align*}
    H'(i,j)
    \;\ldef\;
    \frac{n-1}{n}\, G(\bfB^{ij}) \,+\, \frac{1}{n}\, G(\bfB^{ii}).
  \end{align*}
  By using that $H'(1,2) \leq G(\bfB^{12}) + G(\bfB^{11})$ and applying Lemma~\ref{N_Gaussian} to $\bfB^{12}$ and to $\bfB^{11}$ (with initial measure $\lambda\otimes\lambda$), we get that $\mean\!\big[\exp(cH'(1,2)) \big] \rdef D$ is finite for some constant $c > 0$.  Hence, choosing $C = c[n/2]$, we get
  \begin{align*}
    \prob\!\big[ \bfL^\bfB_n \in K_M^c\big] \;\leq\; \me^{-c[n/2]M}\, D^{[n/2]}.
  \end{align*}
  From this \eqref{eq:exp_tight_modW} follows (up to choosing $K_{3M/c}$ instead of $K_M$).
\end{proof}
\begin{proof}[Proof of Theorem \ref{thm:Sanov_strong_W}]
  The result follows applying the inverse contraction principle (in the version of \cite[Theorem~4.2.4]{DeZe10}), from the space $\cP_1$ to the space $\cP_{(\|\cdot\|+N)^{1+\ve}}$ (with the identity map), having the LDP on the former space (Theorem \ref{thm:Sanov_ebm}) and the exponential tightness on the latter space.
\end{proof}

\section{Extension to $k$-layer enhanced empirical measures}  \label{section:ext_k_layer}
So far we have considered the double layer enhanced empirical measure. We now deal with the extension of the LDP to the $k$-layer enhanced empirical measure, namely
\begin{align}
  \bfL^{B,\{k\}}_n
  \;=\;
  \frac{1}{n^k}\, \sum^n_{i_1, \ldots i_k = 1} \de_{S^{\{kd\}}(B^{i_1}, \ldots, B^{i_k})}.
\end{align}

Here is the extension of Theorem \ref{thm:Sanov_ebm}, that is Theorem \ref{thm:main_res_intro} in the Introduction, for the $1$-Wasserstein metric, extended to a general initial measure.

\begin{theorem}
  \label{thm:Sanov_ebm_k}
  Let $\{B^i : i \in \bbN\}$ be a family of independent $d$-dimensional Brownian motion, with initial measure $\la$ and assume Condition \ref{eq:exp:intcond} for some $c, \ve > 0$. The sequence $\{\mathop{\mathrm{Law}}(\bfL^{\bfB;\{k\}}_n) : n \in \bbN\}$ satisfies a LDP on $\cP_1(\cC^{0,\al}_g([0, T]; \bbR^{kd}))$ endowed with the 1-Wasserstein metric, with scale $n$ and good rate function $\bfI^{\{k\}}$ given by
  \begin{align}
    \bfI^{\{k\}}(\mu)
    \;=\;
    \begin{cases}
      H(\mu\circ\pi_1^{-1}\,|\, P^{\{d\}}), \quad
      &\text{if }\; \mu=F^{\{k\}}(\mu\circ\pi_1^{-1}),\\[1ex]
      \infty,
      &\text{otherwise}. 
    \end{cases}
  \end{align}
\end{theorem}
The proof of this LDP goes like the proof in the double layer case ($k=2$). We recall the main steps and the main changes.

First Lemma~\ref{lemma:enh_true} is extended to the $k$ layer case; the map $Q\mapsto Q^{\otimes k}$ is continuous by Lemma B.4 in the Appendix. Therefore the strategy is still to apply the extended contraction principle. For this we define the approximants $S^{(m);\{k\}}$ as for the double layer case but on $\bbR^{kd}$, the approximation $\bfB^{(m);\{k\}}$ of the enhancement of the $kd$-dimensional Brownian motion and the maps $F^{(m);\{kd\}}$ are defined correspondingly. More in general, the space $\bbR^{2d}$ must be replaced in all the arguments by $\bbR^{kd}$. Then we need to extend Lemmata \ref{lemma:LDP_approx_1} and \ref{lemma:LDP_approx_2} to the $k$ layer case.

\begin{lemma}\label{lemma:LDP_approx_1k}
  For any $\de > 0$, it holds
  \begin{align}
    \lim_{m \to \infty} \limsup_{n\to \infty}\,
    \frac{1}{n}\, \log
    \prob\!\Big[
      d_{W, \al}\big(\bfL^{\bfB;\{k\}}_n,\bfL^{\bfB^{(m)};\{k\}}_n\big) > \de
    \Big] 
    \;=\;
    -\infty.
  \end{align}
\end{lemma}
\begin{proof}
  Consider the coupling measure $\frac{1}{n^k} \sum^n_{i_1, \ldots, i_k=1} \de_{(\bbB^{\{k\}, i_1, \ldots, i_k},\bbB^{(m);\{k\},i_1, \ldots, i_k})}$ with marginals $\bfL^{\bfB^{(m)};\{k\}}_n$ and $\bfL^{\bfB;\{k\}}_n$.  As in the double layer case, we obtain that
  \begin{align}\label{eq:Wasserstein:ub:k}
    d_{W,\al}\big(\bfL^{\bfB;\{k\}}_n,\bfL^{\bfB^{(m)};\{k\}}_n\big)
    &\;\leq\;
    \frac{1}{n^k} \sum^n_{i_1,\ldots i_k=1}\,
    d_{\al}(\bfB^{\{k\}, i_1, \ldots, i_k}, \bfB^{(m);\{k\}, i_1, \ldots, i_k})
    \nonumber\\
    &\;\leq\;
    \frac{(n-k)!}{n!} \sum^n_{\substack{i_1, \ldots, i_k=1\\ \text{mutually distinct}}}\,
    H_{(m);\{k\}}(i_1, \ldots, i_k),
  \end{align}
  where the second sum is obtained from the first one rearranging the terms with at least two equal indices. For instance, in the case $k = 3$, $H_{(m);\{3\}}$ spells out as
  \begin{align*}
    H_{(m);\{3\}}(i_1,i_2,i_3)
    &\;\ldef\;
    \frac{(n-1)(n-2)}{n^2}\, 
    d_{\al}(\bfB^{\{3\},i_1,i_2,i_3}, \bfB^{(m);\{3\},i_1,i_2,i_3})
    \\[.5ex]
    &\mspace{36mu}+\,
    \frac{n-1}{n}\,
    \Big( d_{\al}(\bfB^{\{3\},i_1,i_2,i_1}, \bfB^{(m);\{k\},\{3\},i_1,i_2,i_1})
    \,+\, d_{\al}(\bfB^{\{3\},i_1,i_1,i_3}, \bfB^{(m);\{3\},i_1,i_1,i_3})    \\[.5ex]
    &\mspace{72mu}+
    d_{\al}(\bfB^{\{3\},i_1,i_2,i_2}, \bfB^{(m);\{k\},\{3\},i_1,i_2,i_2}) \Big)
    \\[.5ex]
    &\mspace{36mu}+\,
    \frac{1}{n^2}\, d_{\al}(\bfB^{\{3\},i_1,i_1,i_1}, \bfB^{(m);\{k\},\{3\},i_1,i_1,i_1}).
  \end{align*}
  For a general $k$, $H_{(m);\{k\}}(i_1,\ldots i_k)$ can be written as
  \begin{align}\label{eq:expression_H}
    H_{(m);\{k\}}(i_1, \ldots, i_k)
    &\;=\;
    \frac{n!}{(n-k)!\, n^k}\,
    d_{\al}(\bfB^{\{k\}, i_1, \ldots, i_k}, \bfB^{(m);\{k\},i_1,\ldots i_k})
    \nonumber\\[.5ex]
    &\mspace{36mu}+\;
    \frac{n!}{(n-k)!\,n^k}\,
    \sum_{j_1, \ldots, j_k = 1}^n a_{j_1, \ldots, j_k}\,
    d_{\al}(\bfB^{\{k\}, j_1, \ldots, j_k}, \bfB^{(m);\{k\}, j_1, \ldots, j_k}),
  \end{align}
  where the sum in the second addend is over all $(j_1, \ldots, j_k)$, with at least one repetition of indices, such that, for any $ l \in \{1,\ldots k\}$, there exists $l' \leq l$ with $i_{l'} = j_l$, and the coefficient $a_{j_1, \ldots, j_k}$ is the inverse of a positive integer depending on the repetition of indices in $(j_1,\ldots, j_k)$.  The only relevant fact is that the number of terms is independent of $m$ and $n$ (for fixed $k$) and the coefficients $a_{j_1, \ldots, j_k}$ are bounded by $1$.

  Now, we use Hoeffding's decomposition \cite{Ho63}, for the general $k$ layer case: the right-hand side of \eqref{eq:Wasserstein:ub:k} can be rewritten as
  \begin{align*}
    \frac{1}{n!}\,
    \sum_{\si \in \cS_n} \frac{1}{\lfloor n/k \rfloor}\,
    \sum_{j=1}^{\lfloor n/k \rfloor}
    H_{(m),\{k\}}\big( \si(k j - (k-1)), \ldots, \si(k j)\big),
  \end{align*}
  where $\cS_n$ denotes the set of all permutations of $\{1, \ldots, n\}$. As in the double layer case, an application of the Markov inequality and Jensen's inequality gives, for any $C > 0$ and any $n$ and $m$,
  \begin{align*}
    &\prob\!\Big[
      d_{W,\al}\big(\bfL_n^{\bfB;\{k\}},\bfL_n^{\bfB^{(m)};\{k\}}\big) > \de
    \Big]
    \\[1ex]
    &\mspace{36mu}\leq\;
    \me^{-C \de}\,
    \frac{1}{n!} \sum_{\si \in \cS_n}
    \mean\!%
    \bigg[
      \exp\!%
      \bigg(
      \frac{C}{\lfloor n/k\rfloor}
      \sum^{\lfloor n/k\rfloor}_{i=1}
      H_{(m);\{k\}}\big(\si(k j - (k-1)), \ldots, \si(k j)\big)
      \bigg)
    \bigg].
    \nonumber
  \end{align*}
  By using the mutually independence of $\{H\big(\si(2i-1), \si(2i)\big) : i = 1, \ldots, \lfloor n/2\rfloor\}$, we finally get that
  \begin{align}\label{eq:Wasserstein:exp:ub:k}
    \prob\!\Big[d_{W,\al}\big(\bfL_n^{\bfB}, \bfL_n^{\bfB^{(m)}}\big) > \de \Big]
    \;\leq\;
    \me^{-C \de}\,
    \mean\!%
    \bigg[
      \exp\!\bigg(\frac{C}{\lfloor n/k\rfloor}\, H_{(m);\{k\}}(1,\ldots k)\bigg)
    \bigg]^{\lfloor n/k\rfloor}.
  \end{align}
  On the other hand, using the equality \eqref{eq:expression_H}, we obtain via H\"older inequality
    \begin{align*}
    & \mean\!%
    \bigg[
      \exp\!\bigg(\frac{C}{\lfloor n/k\rfloor}\, H_{(m);\{k\}}(1,\ldots k)\bigg)
    \bigg]^{\lfloor n/k\rfloor}
    \;\leq\;
    \prod^k_{l=1}
    \mean\!\bigg[
    \exp\!\bigg(
      \frac{c_lC}{\lfloor n/k\rfloor}\, d_{\al}(\bfB^{\{l\},i_1,\ldots i_l}, \bfB^{(m);\{l\},i_1,\ldots i_l})
    \bigg)
    \bigg]^{\lfloor c_l'n/k\rfloor}.
    \end{align*}
    Note that the constant $c_l$, $c_l'$ can depend on $l=1,\ldots k$ but are independent of $n$ and $m$, because of the aforementioned uniform bounds on the number of addends and on the coefficients in \eqref{eq:expression_H}. Now, by choosing a suitable $C$, proportional to $m^{\eta} n$ and applying Lemma~\ref{lemma:Strat_approx} to $\bfB^{(m);\{l\},i_1,\ldots i_l}$, $l=1,\ldots k$, we get for a suitable constant $c>0$, for any $\eta \in (0, 1/2 - \al)$ and any $n$ large enough,
  \begin{align*}
    &\sup_{m \geq 1}
    \mean\!%
    \bigg[
      \exp\!%
      \bigg(
        c\frac{m^{\eta} n}{\lfloor n/k\rfloor}\, H_{(m);\{k\}}(1,\ldots k)
      \bigg)
    \bigg]
    \;<\;
    \infty,
  \end{align*}
  By combining this estimate with \eqref{eq:Wasserstein:exp:ub:k}, the assertion follows.
\end{proof}
\begin{lemma}\label{lemma:LDP_approx_2k}
  For every $a < \infty$, it holds
  \begin{align}
    \lim_{m \to \infty}\, \sup_{Q: H(Q\mid P^{\{d\}}) \leq a}\, 
    d_{W,\al}\big(F^{(m);\{k\}})(Q), F^{(m)}(Q)\big)
    \;=\;
    0.
  \end{align}
\end{lemma}
\begin{proof}
  The proof of the lemma goes on as the proof of Lemma~\ref{lemma:LDP_approx_2}. The only changes are: $Q \otimes Q$ must be replaced by $Q^{\otimes k}$, and similarly for the density with respect to $P^{\otimes k}$, $\mu$ must be taken as $(P^{\{d\}})^{\otimes k}$, $h$ as $(\frac{\md Q}{\md P^{\{d\}}})^{\otimes k}$ and the estimate on $\int h\log h\, \md \mu$ becomes $\int h\log h\, \md \mu = k H(Q|P^{\{d\}})\le ka$.
\end{proof}
\begin{proof}[Proof of Theorem~\ref{thm:Sanov_ebm_k}]
  As for the proof of the double layer case, Lemmata \ref{lemma:LDP_approx_1k} and \ref{lemma:LDP_approx_2k} allow to apply the extended contraction principle, which gives the desired result.
\end{proof}
We also have the convergence (in probability) of the enhanced $k$ layer empirical measures, which follows again from the LDP.
\begin{corro}
  The sequence of $\cP_1(\cC^{0,\al}_g([0,T];\bbR^{kd}))$-valued random variables $\{\bfL^{\bfB;\{k\}}_n : n \in \bbN\}$ converges in probability (and in law) to the constant random variable $\bfP^{\{kd\}}$, the enhancement of the $kd$-Wiener measure, that is the law on $\cC^{0,\al}_g([0,T];\bbR^{kd})$ of $(B^{1\ldots k},\bbB^{1\ldots k})$.
\end{corro}
As in the double layer case, the LDP can be extended to the modified Wasserstein topology on $\cC^{0,\al}_g([0,T];\bbR^{kd})$. This completes the proof of Theorem \ref{thm:main_res_intro} as stated in the Introduction.
\begin{theorem}\label{thm:Sanov_strong_W_k}
  The sequence $\mathop{\mathrm{Law}}(\bfL^{\bfB;\{k\}}_n)_n$ satisfies a LDP on $\cP_{(\|\cdot\|+N)^{1+\ve}}(\cC^{0,\al}_g([0,T];\bbR^{kd}))$ (endowed with the $(\|\cdot\|+N)^{1+\ve}$-Wasserstein topology) with scale $n$ and good rate function
  \begin{align}
    \bfI^{\{k\}}(\mu)
    \;=\;
    \begin{cases}
      H(\mu\circ\pi_1^{-1}\,|\, P^{\{d\}}), \quad
      &\text{if }\; \mu = F^{\{k\}}(\mu \circ \pi_1^{-1}),
      \\[1ex]
      \infty,
      &\text{otherwise}. 
    \end{cases}
  \end{align}
\end{theorem}
\begin{proof}
  The proof is analogous to the proof of Theorem \ref{thm:Sanov_strong_W}, we recall only the main points and changes. In all the arguments, $\bbR^{2d}$ must be replaced in all the arguments by $\bbR^{kd}$.

  First, Lemma \ref{N_Gaussian} is extended to $\bfB^{\{k\};i_1, \ldots, i_k}$ for any multiindex $(i_1,\ldots, i_k)$, with a similar proof.  Then, Lemma \ref{lemma:exp_tight_mod_W} is extended to the $k$ layer case and gives the exponential tightness, in the modified Wasserstein topology, of $\mathop{\mathrm{Law}}(\bfL^{\bfB;\{k\}}_n)_n$; the proof is similar to the proof of Lemma \ref{lemma:exp_tight_mod_W}, using the Hoeffding decomposition for the general $k$ layer case (as in the proof of Lemma \ref{lemma:LDP_approx_1k}) and Lemma \ref{N_Gaussian} applied to $\bfB^{\{k\};1,\ldots, k}$ and to $\bfB^{\{k\};i_1, \ldots ,i_k}$ with repetition of indices. The exponential tightness allows to conclude the LDP in the modified Wasserstein topology, by the inverse contraction principle.
\end{proof}

\section{Large deviations for weakly interacting diffusions}   \label{sec:LDP_wid} 
We consider an interacting particle system of the following type:
\begin{align}\label{eq:IPS}
  \left\{
    \begin{array}{rcll}
      \md X^{i,n}_t
      &\mspace{-5mu}=\mspace{-5mu}
      &
      {\displaystyle
        \frac{1}{n}\,
        \sum_{j=1}^n b\big(X^{i,n}_t, X^{j,n}_t\big)\, \md t \,+\, \md B^i_t,
      }
      &\quad i = 1, \ldots n,\\[2ex]
      \mathop{\mathrm{Law}}(X^{i,n}_0)
      &\mspace{-5mu}=\mspace{-5mu}
      &\la
      \qquad \text{i.i.d.}
    \end{array}
  \right.
\end{align}
Here $X^{i,n}$, $i = 1, \ldots, n$ are the unknown positions of the particles, each of them in $\bbR^d$, $b\!:\bbR^d \times \bbR^d \to \bbR^d$ is a given vector field, which we assume regular as much as needed (precisely $C^2_b(\bbR^d\times\bbR^d$), $B^i$ are independent standard $\bbR^d$-valued Brownian motions, on a fixed filtered probability space $(\Om, \cA, (\cF_t)_t, \prob)$, and $\la$ is a given probability measure on $\bbR^d$ satisfying the exponential integrability condition \eqref{eq:exp:intcond} for some $c > 0$, $\ve > 0$. We will omit the superscript $n$ when not necessary. It is well known that the above system admits existence and strong uniqueness (i.e.\ uniqueness for fixed $X_0$ and $B$). \\

The object of interest is an empirical measure associated to this system. For this, let $X^n = (X^{1,n}, \ldots, X^{n,n})$ be the solution to the SDE \eqref{eq:IPS}. However, we will not, as is classical \cite{Sz91}, study  
\begin{align}\label{eq:empmeasIPS}
  L_n^X \;=\; \frac{1}{n}\, \sum_{i=1}^n \de_{X^{i,n}}
\end{align}
but instead, as $n \to \infty$, the \emph{$k$-layer, enhanced} empirical measure $L_n^{\bfX;\{k\}}$, defined in complete analogy to the Brownian motion setting.  To wit, with $k=2$ for notational simplicity only,
\begin{align}\label{eq:enhempmeas}
  \bfL_n^{\bfX, \{2\}} (\om) \equiv  \bfL_n^{\bfX}(\om)  
  \; :=\;
  \frac{1}{n^2}\, \sum_{i,j=1}^n \de_{\bfX^{ij,n}},
\end{align}
where $\bfX^{ij,n} = (X^{ij,n}, \bbX^{ij,n})$ is the rough path on $\bbR^{2d}$ associated with $X^{ij,n} = (X^{i,n}, X^{j,n})$, defined by $\bfX^{ij,n} = S^{\{2d\}}(X^{ij,n})$. Clearly, $L_n^{\bfX;\{2\}}(\om)$ is a (random) measure on $\cC^{0,\al}_g([0,T]; \bbR^{2d})$ and we define, on the space of such measures,\footnote{Given $b\!: \bbR^{2d} \to \bbR^d$, using notation $(x^1, x^2) \in \bbR^d \times \bbR^d = \bbR^{2d}$, we denote by $\brb\!:\bbR^{2d} \to \bbR^{2d}$ the function such that $\brb(x^1,x^2)^1 = b(x^1,x^2)$ and $\brb(x^1,x^2)^2 = 0$. } 
\begin{align}\label{eq:defK}
  \bfK_b(\mu)
  &\;=\;
  \int_{\cC^{0,\al}_g([0,T];\bbR^{2d})} \int^T_0 \brb(X_t)\, \md\bfX_t\, \mu(\md\bfX)
  \nonumber\\[1ex]
  &\mspace{36mu}-\;
  \frac{1}{2}\,
  \int^T_0
    \int_{\cC^{0,\al}_g([0,T];\bbR^{2d})} \diverg \brb(X_t)\, \mu(\md\bfX)\,
  \md t
  \\[1ex]
  &\mspace{36mu}-\;
  \frac{1}{2}\,
  \int^T_0
    \int_{\cC^{0,\al}_g(\bbR^{2d})}
      \bigg(\int_{\cC^{0,\al}_g(\bbR^{2d})}\brb(X^1_t,Y^2_t)\, \mu(\md\bfY)\bigg)^2\,
    \mu(\md\bfX)\,
  \md t,
  \nonumber
\end{align}
noting that $\bfK_b(\mu)$ is well-defined whenever $b \in C^2_b$ and $\mu \in \cP_{(\|\cdot\|+N)^{1+\ve}}(\cC^{0,\al}_g([0, T];\bbR^{2d}))$.

Call $\Pi_{2}$ the projection from $G^2 (\bbR^{kd}) \to G^2 (\bbR^{2d})$\footnote{$\Pi_2$ is the projection $(X^{12}, \ldots; \bbX^{12}, \ldots) \mapsto (X^{12}; \bbX^{12})$.}. Given a measure $\mu$ on $\cC_g^{0,\al}([0,T]; \bbR^{kd})$, the image measure $(\Pi_2)_* \mu \equiv \mu \circ \Pi_2^{-1}$ is a measure on the $2$-layer rough path space $\cC_g^{0,\al}([0,T]; \bbR^{2d})$. As previously, $P^{\{d\}}$ is $d$-dimensional Wiener measure with $\lambda$ initial distribution. $N$ was introduced in (\ref{def:N}).
\begin{theorem}\label{generalres}
  Assume that $b$ is in $C^2_b(\bbR^d \times \bbR^d)$, let $X^n = (X^{1,n}, \ldots, X^{n,n})$ be the solution to the system \eqref{eq:IPS}, with initial law $\la$ satisfying Condition \eqref{eq:exp:intcond} for fixed $\ve >0$, and let $\bfL^{\bfX,\{k\}}_n$ be the corresponding enhanced $k$-layer empirical measure, $k \ge 2$.  Fix $\al \in (1/3,1/2)$. Then, the sequence of laws $\{\mathop{\mathrm{Law}}(\bfL^\bfX_n) : n \in \bbN\}$ satisfies a large deviation principle on $\cP_{(\|\cdot\|+N)^{1+\ve}}(\cC^{0,\al}_g([0, T];\bbR^{kd}))$ with scale $n$ and good rate function $\bfJ_b$ given by
  \begin{align}
    \bfJ_b^{\{k\}}(\mu) \equiv \bfJ_b(\mu)   
    \;=\;
    \begin{cases}
      H(\mu\circ\pi_1^{-1}|P^{\{d\}})-\bfK_b(\mu \circ \Pi_2^{-1}), \quad      &\text{if }\; \mu=F^{\{k\}}(\mu\circ\pi_1^{-1}),\\[1ex]
      \infty,
      &\text{otherwise}. 
    \end{cases}
  \end{align}
\end{theorem}

\begin{proof}
  \textsc{First step}: Enhanced Girsanov theorem. Let $X=X^n = (X^{1,n}, \ldots, X^{n,n})$ be the solution to the SDE \eqref{eq:IPS} with Stratonovich lift $\bfX=S^{\{nd\}}(X)$. We prove that the law of $\bfX$ on $\cC^{0,\al}_g([0,T];\bbR^{nd})$ is absolutely continuous with respect to the law of the enhanced $nd$-dimensional Brownian motion $\bfB$, with density given by $\exp(\rho_n(\bfB))$, where $\rho_n$ is deterministically defined by (recall $b=b(x,y)$)
  \begin{align}\label{eq:Girsanov_density}
    \rho_n(\bfY)
    &\;=\;
    \frac{1}{n}\, \sum^n_{i,j = 1}\,
    \int_0^T \brb(Y^{ij}_t) \cdot \md\bfY^{ij}_t
    \,-\,
    \frac{1}{2 n}\, \sum^n_{i,j = 1}\,
    \int_0^T \diverg \brb(Y^{ij}_t)\, \md t
    \,-\,
    \frac{1}{2 n}\, \sum^n_{i = 1}\,
    \int_0^T \diverg_y b(Y^{ii}_t)\, \md t
    \nonumber\\[1ex]
    &\mspace{36mu}
    \,-\,
    \frac{1}{2}\, \sum^n_{i=1}\,
    \int_0^T \bigg| \frac{1}{n}\, \sum^n_{j=1} \brb(Y^{ij}_t)\, \bigg|^2\, \md t.
  \end{align}
  Indeed, the classical Girsanov theorem applied to \eqref{eq:IPS}, combined with Proposition \ref{rough_Strat_int}, gives that, for every $\psi$ measurable bounded function on $C^\alpha([0,T];\bbR^{nd})$,
  \begin{align*}
    \mean\!\big[ \psi(X) \big] \;=\; \mean\!\big[\me^{\rho_n(\bfB)}\psi(B)\big].
  \end{align*}
  By applying the previously obtained formula to $\psi(X) \ldef \vp(S^{\{nd\}}(X))$, where $\vp$ is any measurable bounded function on $\cC^{0,\al}([0,T];\bbR^{nd})$, we get
  \begin{align}\label{eq:densityenhGirsanov}
    \mean\!\big[\vp(\bfX)\big]
    \;=\;
    \mean\!\big[\me^{\rho_n(\bfB)}\vp(\bfB)\big],
  \end{align}
  that is enhanced Girsanov theorem.
  
  \textsc{Second step}: Density for the law of the enhanced empirical measures. First consider the double-layer case $k=2$. We prove that on the space $\cP_{(\|\cdot\|+N)^{1+\ve}}(\cC^{0,\al}_g([0,T];\bbR^{2d}))$ the law of the enhanced empirical measure $\bfL^\bfX_n$ is absolutely continuous with respect to the law of $\bfL^\bfB_n$, with density given by $\exp(n\bfK_b)\exp(\bfK_b')$ for a bounded function $\bfK_b'$ specified below. The main point is that
  \begin{align}\label{eq:MF_Girsanov}
    \rho_n(\bfB)
    \;=\;
    n\, \bfK_b(\bfL_n^{\bfB}) \,+\, \bfK_b'(\bfL_n^{\bfB})
    \;=\;
    n\, \bfK_b(\bfL_n^{\bfB;\{2\}}) \,+\, \bfK_b'(\bfL_n^{\bfB;\{2\}}),
  \end{align}
  where
  \begin{align*}
    \bfK_b'(\mu)
    \;=\;
    -\int_0^T \int_{\cC^{0,\alpha}_g([0,T];\bbR^{2d})}
    \diverg_y b(\pi_1(\bfY_t),\pi_1(\bfY_t))\, \mu(\md \bfY)\, \md t.
  \end{align*}
  This follows from formula \eqref{eq:Girsanov_density}, the structural reason being the mean field interaction. Now by Lemma \ref{lemma:RP_cont_empmeas} in the Appendix (applied with $k=2$) the enhanced empirical measure associated with a rough path in $\bbR^{nd}$ is a continuous, in particular measurable function $G_n$ of the rough path, that is $\bfL_n^\bfX = G_n(\bfX)$, $\bfL_n^\bfB = G_n(\bfB)$.  So it is enough to apply formula \eqref{eq:densityenhGirsanov} to $\vp = \phi \circ G_n$, where $\phi$ is any measurable bounded function on $\cP_{(\|\cdot\|+N)^{1+\ve}}(\cC^{0,\al}_g([0, T];\bbR^{2d}))$, and to use \eqref{eq:MF_Girsanov}.
  
  In the case $k>2$, on the space $\cP_{(\|\cdot\|+N)^{1+\ve}}(\cC^{0,\al}_g([0,T];\bbR^{kd}))$ the law of the enhanced empirical measure $\bfL^{\bfX,\{k\}}_n$ has density (with respect to the law of $\bfL^{\bfB,\{k\}}_n$ given by $\exp(n\bfK_b\circ\Pi_2)\exp(\bfK_b'\circ\Pi_2)$. Indeed,
  \begin{align*}
    (\Pi_2)_* \bfL_n^{\bfB;\{k\}}
    \;=\;
    \bfL_n^{\bfB;\{k\}} \circ \Pi_2^{-1}
    \;=\;
    \bfL_n^{\bfB;\{2\}},
  \end{align*}
  and therefore
  \begin{align*}
    \rho_n(\bfB)
    \;=\;
    n\, \bfK_b(\bfL_n^{\bfB;\{k\}} \circ \Pi_2^{-1})
    \,+\,
    \bfK_b'(\bfL_n^{\bfB;\{k\}} \circ \Pi_2^{-1})
    \;=\;
    n\, \bfK_b(\bfL_n^{\bfB;\{2\}}) \,+\, \bfK_b'(\bfL_n^{\bfB;\{2\}}).
  \end{align*}
  We can conclude as in the double-layer case (applying Lemma \ref{lemma:RP_cont_empmeas} to the general $k$ layer case).

  \textsc{Third step}: LDP for $Z_n^{-1}\exp(n\bfK_b\circ\Pi_2^{-1})\mathop{\mathrm{Law}}(\bfL_n^{\bfB;\{k\}})$ (and goodness of $\bfJ_b$). We are ready to prove a large deviation principle for the family $Z_n^{-1}\exp(n\bfK_b)\mathop{\mathrm{Law}}(\bfL_n^{\bfB;\{k\}})$, where $Z_n = \mean[\exp(n\bfK_b(\bfL_n^{\bfB;\{2\}}))]$ is the usual renormalization constant. Indeed, the second step invites to apply Varadhan lemma (Theorem \ref{Varadhan_lemma}, which is an easy and well-known consequence of Varadhan lemma in \cite[Theorem~4.3.1]{DeZe10}).  We need to verify the hypotheses, namely, for $k=2$, that $\bfK_b$ is a continuous function on $\cP_{(\|\cdot\|+N)^{1+\ve}}(\cC^{0,\al}_g([0,T]; \bbR^{2d}))$ and that it holds, for some $\ga > 1$,
  \begin{align}\label{eq:condVaradhan}
    \limsup_{n \to \infty}\,
    \frac{1}{n}\,
    \log \mean\!\big[\exp\!\big(n \ga \bfK_b(\bfL^\bfB_n) \big) \big]
    \;<\;
    \infty;
  \end{align}
  The hypotheses for general $k$ follow from those for $k=2$ (so we will fix and omit $k=2$ in the argument below).
  
  On the continuity of $\bfK_b$, it is easy to see that the deterministic integrals in formula \eqref{eq:defK} (i.e.\ the second and third addend) are continuous bounded functions of $\mu$ (they are actually continuous bounded functions of $Q=\mu\circ\pi_1^{-1}$ in the $C_b$-weak topology on $\cP(C^{0,\al}([0,T];\bbR^d))$), so we concentrate on the term with the rough integral.  By Theorem~\ref{thm:cont_rough_int}, the rough integral
  \begin{align}
    I_b(\bfX) \;\ldef\; \int^T_0 \brb(X)\, \md\bfX
  \end{align}
  is continuous on $\cC^{0,\al}_g$ with at most linear growth with respect to $N$ (by \eqref{eq:lineargrowthRP}). So, by Proposition \ref{Nlingrowth}, the term
  \begin{align*}
    \int_{\cC^{0,\al}_g([0, T];\bbR^{2d})} \int^T_0 \brb(X_t)\, \md\bfX_t\, \mu(\md\bfX)
    \;=\;
    \int_{\cC^{0,\al}_g([0, T]; \bbR^{2d})} I_b\, \md\mu
  \end{align*}
  is continuous on $\cP_{(\|\cdot\|+N)^{1+\ve}}(\cC^{0,\al}_g([0,T]; \bbR^{2d}))$.  Hence $\bfK_b$ is continuous.  Now we prove \eqref{eq:condVaradhan} with $\gamma=2$. We use the fact that
  \begin{align*}
    M_t
    \;=\;
    \exp\!\bigg(
      2n\, \bfK_b(\bfL^\bfB_n)
      \,-\,
      \sum^n_{i=1}\,
      \int_0^T \frac{2}{n^2}\, \bigg| \sum^n_{j=1} b(X^i-X^j) \bigg|^2\, \md t
    \bigg)
  \end{align*}
  is a martingale, as one can verify easily (and classically). Hence we have
  \begin{align*}
    \mean\!\Big[\exp\!\big(2n\bfK_b(\bfL^\bfB_n)\big)\Big]
    &\;=\;
    \mean\!\bigg[
      M_T\,
      \exp\!\bigg(
        \sum^n_{i=1}\,
        \int_0^T \frac{2}{n^2}\, \bigg| \sum^n_{j=1} b(X^i-X^j) \bigg|^2\, \md t
      \bigg)
    \bigg]
    \\[1ex]
    &\;\leq\;
    \me^{2nT \|b\|_\infty}\, \mean\!\big[M_T \big]
    \;=\;
    \me^{2nT \|b\|_\infty},
  \end{align*}
  which implies \eqref{eq:condVaradhan}.  Hence, we can apply Varadhan lemma and get the LDP for $\{Z_n^{-1}\exp(n\bfK_b\circ\Pi_2^{-1})\mathop{\mathrm{Law}}(\bfL_n^{\bfB;\{k\}}) : n \in \bbN\}$ with rate function $\bfJ_b$.  Moreover $\bfJ_b$ is good: this follows from exponential tightness of $\{\mathop{\mathrm{Law}}(\bfL^{\bfB;\{k\}}_n) : n \in \bbN\}$ and Varadhan lemma (in the version \ref{Varadhan_lemma}).
  
  \textsc{Conclusion}. In order to conclude the LDP for $\{\mathop{\mathrm{Law}}(\bfL^{\bfX;\{k\}}_n) : n \in \bbN\}$, note that
  \begin{align*}
    \{\mathop{\mathrm{Law}}(\bfL^{\bfB;\{k\}}_n) : n \in \bbN\}
    \;=\;
    Z_n\,
    \exp(\bfK_b'\circ\Pi_2^{-1})\,
    \big(
      Z_n^{-1}\, \exp(n\bfK_b\circ\Pi_2^{-1})\,
      \mathop{\mathrm{Law}}(\bfL_n^{\bfB;\{k\}})
    \big).
  \end{align*}
  Therefore, for any Borel set $A$ in $\cP_{(\|\cdot\|+N)^{1+\ve}}(\cC^{0,\al}_g([0,T]; \bbR^{kd}))$, we have
  \begin{align*}
    \frac{1}{n}\, \log \prob\!\big[\bfL^{\bfB;\{k\}}_n\in A \big]
    &\;\leq\;
    \limsup_n \frac{1}{n}\, \log
    \big(
      Z_n^{-1}\, \exp(n\bfK_b\circ\Pi_2^{-1})\,
      \mathop{\mathrm{Law}}(\bfL_n^{\bfB;\{k\}})(A)
    \big)
    \\[.5ex]
    &\mspace{36mu}+\;
    \limsup_n \frac{1}{n}\, \sup_{\mu\in A} \big|\bfK_b'(\Pi_2^{-1}(\mu))\big|
    \,+\,
    \limsup_n \frac{1}{n}\, \big|\log Z_n\big|.
  \end{align*}
  Now, $\bfK_b'$ is bounded on the whole $\cP_{(\|\cdot\|+N)^{1+\ve}}(\cC^{0,\al}_g([0,T]; \bbR^{kd}))$. Moreover,
  \begin{align*}
    Z_n
    \;=\;
    \mean\!\big[%
      \exp(n\bfK_b(\bfL_n^{\bfB;\{2\}}))
    \big]
    \;=\;
    \mean\!\big[%
      \exp(-\bfK_b'(\bfL_n^{\bfB;\{2\}}))\, \exp(n\bfK_b(\bfL_n^{\bfB;\{2\}})
      \,+\,
      \bfK_b'(\bfL_n^{\bfB;\{2\}}))
    \big],
  \end{align*}
  since $\mean[\exp(n\bfK_b(\bfL_n^{\bfB;\{2\}}) + \bfK_b'(\bfL_n^{\bfB;\{2\}}))] = 1$ (the exponential being a density) and $\bfK_b'$ is bounded from above and from below, $|\log Z_n|$ is bounded uniformly in $n$. Hence,
  \begin{align*}
    \frac{1}{n}\, \log \prob\!\big[\bfL^{\bfB;\{k\}}_n\in A \big]
    \;\leq\;
    \limsup_n \frac{1}{n}\, \log
    \big(
      Z_n^{-1}\, \exp(n\bfK_b\circ\Pi_2^{-1})\,
      \mathop{\mathrm{Law}}(\bfL_n^{\bfB;\{k\}})(A)
    \big);
  \end{align*}
  similarly (with reverse inequalities) for the $\liminf_n$. Then the LDP for $\{\mathop{\mathrm{Law}}(\bfL^{\bfX;\{k\}}_n) : n \in \bbN\}$ follows from that for $\{Z_n^{-1}\exp(n\bfK_b\circ\Pi_2^{-1})\mathop{\mathrm{Law}}(\bfL_n^{\bfB;\{k\}}) : n \in \bbN\}$. The proof is complete.
\end{proof}
We insist that it is crucial in the above proof to work with $k=2$ (or more) layers, for otherwise the argument - based on continuity of $\bfK$ - fails. Theorem \ref{generalres} implies, of course, an immediate LPD for the ($1$-layer, non-enhanced) empircal measure  $L_n^X$ as defined in \ref{eq:empmeasIPS}: it suffices to apply the contraction principle, applied to the map
\begin{align*}
  \cP_{(\|\cdot\|+N)^{1+\ve}}(\cC^{0,\al}_g([0,t];\bbR^{2d}))
  \;\ni\;
  \nu
  \;\longmapsto\;
  \nu \circ \pi_1^{-1}
  \;\in\;
  \cP_1(C^{0,\al}([0,T];\bbR^d)),
\end{align*}
with resulting (good) rate function
\begin{align*}
  J_b(Q) \;\ldef\l \inf\{\bfJ_b(\mu) : \mu \circ \pi_1^{-1} = Q\}.
\end{align*}
The (only) purpose of the following corollary is to re-express this rate function in more familiar terms of stochastic analysis. To this end, we define, for any measure $Q$ on $C^{0,\alpha}([0,T];\bbR^d)$ which makes the coordinate process, and then also the doubled coordinate process $X = (X^1, X^2)$ under $Q \otimes Q$, a Wiener process plus a square integrable (in time and $\Omega$) drift (this happens when $H(Q|P^{\{d\}}) < \infty$, see the proof of Corollary \ref{mainres}),
\begin{align}\label{eq:def_b}
  K_b(Q)
  &\;\ldef\;
  \int_{C^{0,\alpha}([0,T];\bbR^{2d})}
    \int_0^T \brb(X_t) \circ \md X_t\,
  (Q \otimes Q)(\md X)
  \nonumber\\[1ex]
  &\mspace{36mu}-\;
  \frac{1}{2}\,
  \int_0^T
    \int_{C^{0,\alpha}([0,T];\bbR^{2d})} \diverg\, \brb(X_t)\, (Q\otimes Q)(\md X)
  \md t
  \\[1ex]
  &\mspace{36mu}-\;
  \frac{1}{2}\,
  \int^T_0
    \int_{C^{0,\alpha}([0,T];\bbR^d)}
      \bigg| \int_{C^{0,\alpha}([0,T];\bbR^d)} \brb(Y_t,Z_t)\, Q(\md Z) \bigg|^2\,
    Q(\md Y)\,
  \md t.
  \nonumber
\end{align}
Note the last two summands (integrals against $\md t$) are finite under our assumptions on $b$.
\begin{corollary}\label{mainres}
  Under the assumptions of Theorem \ref{generalres}, the sequence of laws $\{\mathop{\mathrm{Law}}(L^X_n) : n \in \bbN\}$ satisfies an LDP on $\cP_1(C^{0,\al}([0,T];\bbR^d))$ with scale $n$ and good rate function $J_b$ given by
  \begin{align}
    J_b(Q) \;=\; H(Q \,|\, P^{\{d\}}) \,-\, K_b(Q),
  \end{align}
  with the understanding that the right-hand side above is $+\infty$ whenever $H(Q \,|\, P^{\{d\}})=+\infty $.
\end{corollary}
\begin{proof}
  Consider a measure $Q$ with $H(Q,P^{\{d\}}) = \infty$.  We need to show that $\inf\{\bfJ_b(\mu) : \mu \circ \pi_1^{-1} = Q\} = \infty$, that is, $\bfJ_b(\mu) = \infty$ whenever $\mu$ projects to $Q$.  By looking at the definition of $\bfJ_b$, there is nothing to show unless $\mu = F^{\{2\}}(\mu \circ \pi_1^{-1}) = F^{\{2\}}(Q)$. But in this case $\bfJ_b(\mu) = H(Q|P^{\{d\}})-\bfK_b(\mu) = \infty$, as desired.
  
  We now consider a measure $Q$ with $H(Q,P^{\{d\}}) < \infty$.  We have to show that
  \begin{align*}
    H(Q \,|\, P^{\{d\}}) \,-\, K_b(Q)
    \;=\;
    \inf\{\bfJ_b(\mu) : \mu \circ \pi_1^{-1} = Q\}.
  \end{align*}
  In fact, from the very definition of $\bfJ_b$, we have $\bfJ_b (\mu) = \infty$ unless $\mu = F^{\{2\}} (Q)$. This measure $\mu=F^{\{2\}}(Q)$ satisfies $Q = \mu \circ \pi_1^{-1} $: by Proposition \ref{prop:conv_S}, denoting $Y = (Y^1,Y^2)$ the canonical process on $C^{0,\al}([0,T];\bbR^{2d})$, $P^{\{d\}} \otimes P^{\{d\}} \simeq P^{\{2d\}}$-a.s.,\ and so $Q\otimes Q$-a.s.,\ $S^{m,\{2\}}(Y)$ converges to $S^{\{2\}}(Y) = ((Y^1,Y^2),\bbY)$, in particular $\pi_1(S(Y)) = Y^1$ $Q\otimes Q$-a.s.\ and so $F^{\{2\}}(Q) \circ \pi_1^{-1} = Q$.  For such a $\mu$, we have $\bfJ_b(\mu) =   H( \mu \circ \pi_1^{-1} |P^{\{d\}})-\bfK_b(\mu) = H (Q | P^{\{d\}}) - \bfK_2 (F^{\{2\}} (Q))$. Thus, it only remains to see that
  \begin{align*}
    K_b(Q) \;=\; \bfK_b( F^{\{2\}} (Q) ).
  \end{align*}
  Since $H(Q,P^{\{d\}}) < \infty$, by a classical result (see for example \cite[Section II Remark 1.3]{Fo88}), there exists an adapted process $g$ such that $W_t = X_t - X_0 - \int^t_0 g_r\, \md r$ is a Wiener process under $Q$ and, denoting by $\nu$ the marginal of $Q$ at time $0$, it holds
  \begin{align}\label{eq:entropy_Girsanov}
    H(Q\,|\,P)
    \;=\;
    H(\nu \,|\, \la)
    \,+\,
    \frac{1}{2}\, \Mean^Q\!\bigg[\int^t_0 |g_r|^2\, \md r\bigg].
  \end{align}
  In particular we can define $\int^t_s \brb(X_r) \circ\md X_r$, which appears in the definition of $K_b$ (so $K_b(X)$ makes sense), as a Stratonovich integral under $Q\otimes Q$ or equivalently under $P\otimes P$, and by Proposition~\ref{rough_Strat_int} this integral coincides $P\otimes P$-a.s.\ (and so $Q\otimes Q$-a.s.)\ with the rough integral $\int^T_0 \brb(X_t)\, \md\bfX_t$ in the definition of $\bfK_b$. Therefore,
  \begin{align*}
    \int_{\cC^{0,\al}_g([0,T];\bbR^{2d})}
    \int^T_0 \brb(X_t)\, \md\bfX_t\, \mu(\md\bfX)
    \;=\;
    \int_{C^{0,\alpha}([0,T];\bbR^{2d})}\,
    \int_0^T \brb(X_t) \circ \md X_t\, (Q \otimes Q)(\md X),
  \end{align*}
  i.e.\ the first addends in the definitions of $K_b(Q)$ and $\bfK_b(\mu)$ coincide. The other addends also coincide, as easily verified (they are classical integrals). Therefore $K_b(Q) = \bfK_b( F^{\{2\}} (Q) )$ as desired.
\end{proof}
The above discussion has another useful consequence.
\begin{lemma} \label{lem:bfJisJ}
  The rate function given in Theorem \ref{generalres} satisfies
  \begin{align}
    \bfJ_b^{\{k\}}(\mu) \;=\; J_b ( Q)  
  \end{align}
  whenever $\mu = F^{\{k\}}(Q)$ and infinite otherwise.
\end{lemma}

\section{Application 1: Robust propagation of chaos } \label{rPOC}
It is an elementary fact of large deviations theory, that a LDP at scale $n$ with good rate function, which has a \emph{single} zero, implies a (weak and in fact -- thanks to Borel--Cantelli -- strong) law of large numbers.  We now give different representations of the rate functions obtained in the last section, which will allow to ``see'' the single zero. This requires us to consider the following mean field (McKean--Vlasov) SDE on $\bbR^d$
\begin{align}\label{eq:meanfieldSDE}
  \left\{
    \begin{array}{rcl}
      \md \bar{X}_t
      &\mspace{-5mu}=\mspace{-5mu}
      &
      {\displaystyle
        (b*u_t)(\bar{X}_t)\, \md t \,+\, \md\bar{B}_t
      }
      \\[.5ex]
      u_t
      &\mspace{-5mu}=\mspace{-5mu}
      &\mathop{\mathrm{Law}}(\bar{X}_t)
      \\[.5ex]
      \la
      &\mspace{-5mu}=\mspace{-5mu}
      &\mathop{\mathrm{Law}}(\bar{X}_0)
    \end{array}
  \right.
\end{align}
where
\begin{align}
  (b*u_t)(x) \; \ldef\; \int_{\bbR^d} b(x ,y)\, u_t(\md y).
\end{align}
The law $\prob^{\bar{X}}$ of the solution $\bar{X}$ can be seen as fixed point of the map $\Phi$ defined in this way: for any probability measure $Q$ on $C^{0,\al}([0,T],\bbR^d)$, calling $Q_t$ the marginal of $Q$ at time $t$, $\Phi(Q)$ is the law of the solution to the SDE
\begin{align}
  \left\{
    \begin{array}{rcl}
      \md \bar{Y}_t
      &\mspace{-5mu}=\mspace{-5mu}
      &
      {\displaystyle
        (b*Q_t)(\bar{Y}_t)\, \md t \,+\, \md\bar{B}_t
      }
      \\[.5ex]
      \la
      &\mspace{-5mu}=\mspace{-5mu}
      &\mathop{\mathrm{Law}}(\bar{X}_0)
    \end{array}
  \right.
\end{align}
(for given $X_0$ and $Q$, this SDE has a pathwise-unique solution).  
\begin{lemma} \label{lem:zeros} 
  \begin{itemize}
  \item[(i)] (``one-layer, non-enhanced'') The zeros of $J_b$ are precisely fixed points of $\Phi$, as is seen from
    \begin{align*}
      J_b(Q) \;=\; H(Q \,|\, \Phi(Q)). 
    \end{align*}
  \item[(ii)] The zeros of $\bfJ^{\{k\}}_b$ are precisely the image under $F^{\{k\}}$ of fixed points of $\Phi$, as is seen from Lemma~\ref{lem:bfJisJ}.
  \end{itemize}
\end{lemma}
\begin{proof}
  (i) Indeed, by Girsanov theorem, $\Phi(Q)$ is absolutely continuous with respect to $P^{\{d\}}$, with density satisfying
  \begin{align*}
    \log\frac{\md\Phi(Q)}{\,d P^{\{d\}}}(X)
    &\;=\;
    \int_{C^{0,\alpha}([0,T];\bbR^d)} \int^T_0 \brb(X_t,Z_t) \circ \md X_t\; Q(\md Z)
    \\[1ex]
    &\mspace{36mu}-\;
    \frac{1}{2}\,
    \int^T_0
      \int_{C^{0,\alpha}([0,T];\bbR^d)} \diverg \brb(X_t,Z_t)\, Q(\md Z)\,
    \md t
    \\[1ex]
    &\mspace{36mu}-\;
    \frac{1}{2}\,
    \int^T_0
        \bigg| \int_{C^{0,\alpha}([0,T];\bbR^d)} \brb(X_t,Z_t)\, Q(\md Z)\bigg|^2\,
    \md t
  \end{align*}
  (where we have used stochastic Fubini theorem for exchanging stochastic integration and integration in $Q$ in the first term). Notice that, for $Q$ absolutely continuous with respect to $P^d$,
  \begin{align*}
    K_b(Q)
    \;=\;
    \int_{C^{0,\al}([0,T];\bbR^d)} \log \frac{\md\Phi(Q)}{\md P^{\{d\}}}(X)\, Q(\md X),
  \end{align*}
  so we have
  \begin{align*}
    J_b(Q)
    \;=\;
    \int_{C^{0,\al}([0,T];\bbR^d)} \log \frac{\md Q}{\md P^{\{d\}}}\, \md Q
    \,+\,
    \int_{C^{0,\al}([0,T];\bbR^d)} \log \frac{\md P^{\{d\}}}{\md \Phi(Q)}\, \md Q
    \;=\;
    H(Q \,|\, \Phi(Q)).
  \end{align*}

  (ii) The statement follows from Part (i) and Lemma \ref{lem:bfJisJ}.
\end{proof}
The previous lemma applies nicely in view of the following result (well-known, see e.g.\ \cite[Theorem 1.1 in Chapter 1]{Sz91}).
\begin{prop}
  For $b$ as in Theorem \ref{generalres}, there is a unique strong solution to (\ref{eq:meanfieldSDE}) and its law $\mathop{\mathrm{Law}}(\bar{X})$ is the unique fixed point of $\Phi$.
\end{prop} 
We hence know that $J_b$, although not necessarily convex, has exactly one zero, given by $\bbP^{\bar{X}} = \mathop{\mathrm{Law}} (\bar{X})$. Similarly, and more importantly, $\bfJ_b^k$ has exactly one zero given by 
\begin{align*}
  F^{\{k\}} ( \mathop{\mathrm{Law}} (\bar{X}))
  \;=\;
  ( \mathop{\mathrm{Law}} (\bar{X})^{\otimes{k}}) \circ (S^{\{kd\}})^{-1}.
\end{align*}
This law is, of course, nothing else that the law of the Stratonovich lift of $k$ IID copies $\bar{X}^{1},\dots ,\bar{X}^{k}$ of the McKean--Vlasov diffusion $\bar{X}$.  We can now deduce the enhanced propagation of chaos result as stated in the introduction. 
\begin{theorem} 
  Under the assumptions of Theorem \ref{generalres} (that is, $b \in C^2_b$) and for all integer $k$,
  \begin{align}
    \mathop{\mathrm{Law}}\big( S^{\{kd\}}( X^{1,n}, \dots, X^{k,n}) \big)    \underset{n \to \infty }{\;\longrightarrow\;}
    \mathop{\mathrm{Law}}\big( S^{\{kd\}}( \bar{X}^{1}, \dots ,\bar{X}^{k}) \big),
  \end{align}
  as $C_b$-weak convergence of probability measures on $\mathcal{C}^{0,\alpha}_g( [0, T]; \bbR^{kd}) $ equipped with $\al$-H\"{o}lder rough path topology.
\end{theorem}
\begin{proof}
  Theorem \ref{generalres} gives us a LDP that quantifies the convergence (in the $C_b$-weak topology on $\cP(\cC^{0,\al}_g([0, T]; \bbR^{kd}))$) in probability
  \begin{align*}
    \frac{1}{n^{k}}\,
    \sum_{i_{1}, \ldots, i_k = 1}^n\, \de_{S(X^{i_{1},n}(\omega), \dots, X^{i_{k},n}(\omega))}
    \underset{n \to \infty}{\;\longrightarrow\;}
    \mathop{\mathrm{Law}}\big(S_{2}(\bar{X}^{1}, \dots, \bar{X}^{k}) \big).
  \end{align*}
  This convergence follows (by standard reasoning in large deviations on metric spaces) from Theorem~\ref{generalres} and Lemma~\ref{lem:zeros} which identifies the law of $S_{2}(\bar{X}^{1}, \dots, \bar{X}^{k})$ as the unique zero of the rate function.  Testing against
  \begin{align*}
    \varphi \;\in\; C_{b}\big( \cC^{0,\alpha}_g([0, T]; \bbR^{kd})\big)
  \end{align*}
  we get convergence (in probability, as $n \to \infty$)
  \begin{align*}
    \frac{1}{n^{k}}\!
    \sum_{i_{1}, \ldots, i_k = 1}\mspace{-6mu}
    \varphi\big( S_{2}( X^{i_{1},n}(\omega), \dots, X^{i_{k},n}(\omega)) \big)
    &\underset{n \to \infty}{\;\longrightarrow\;}
    \big\langle
      \varphi,\,
      \mathop{\mathrm{Law}}\big(S_{2}(\bar{X}^{1}, \dots, \bar{X}^{k}) \big)
    \big\rangle
    \;=\;
    \Mean\!\big[\varphi\big(S_{2}(\bar{X}^{1}, \dots, \bar{X}^{k})\big) \big].
  \end{align*}
  Now take expectation $\mean[\cdot]$ on both sides.  By using the boundedness of $\varphi$,
  \begin{align*}
    &\frac{1}{n^{k}}\!
    \sum_{i_1, \ldots, i_k = 1}^n\mspace{-6mu}
    \mean\big[
      \varphi \big(S_{2}( X^{i_{1},n}(\omega), \dots, X^{i_{k},n}(\omega) ) \big)
    \big]
    \\[1ex]
    &\mspace{36mu}=\;
    \frac{1}{n^{k}}\!
    \sum_{\substack{i_1, \ldots, i_k=1\\ \text{mutually distinct}}}^n\mspace{-24mu}
    \mean\!\big[
      \varphi\big(S_{2}(X^{i_{1},n}(\omega), \dots, X^{i_{k},n}(\omega) ) \big)
    \big]
    \,+\, \cO(1/n)
    \\[1ex]
    &\mspace{36mu}=\;
    \frac{n!}{n^k (n-k)!}\,
    \Mean\!\big[
      \varphi\big(S_{2}(X^{1,n}(\omega), \dots, X^{k,n}(\omega) ) \big)
    \big]
    \,+\, \cO(1/n),
  \end{align*}
  we have
  \begin{align*}
    \Mean\!\big[
      \varphi\big(S_{2}(X^{1,n}(\omega), \dots, X^{k,n}(\omega) ) \big)
    \big]
    \underset{n \to \infty}{\;\longrightarrow\;}
    \Mean\!\big[
      \varphi\big(S_{2}(\bar{X}^1, \dots, \bar{X}^k ) \big)
    \big].
  \end{align*}
  Since $\varphi \in C_{b}( \cC^{0,\alpha}_g([0, T]; \bbR^{kd}))$ is arbitrary, we proved (in rough path topology!), that
  \begin{align*}
    S_{2}\big( X^{1,n}(\omega), \dots, X^{k,n}(\omega) \big)
    \;\Longrightarrow\;
    S_{2}\big( \bar{X}^{1}, \dots, \bar{X}^{k}\big)
    \qquad \text{as }\; n \to \infty.
  \end{align*}
\end{proof}

\begin{remark}
As already noted in the Introduction, this enhanced propagation of chaos is also a consequence of classical propagation of chaos and classical It\^o calculus, applying It\^o formula to $\int^t_s X^{i,n}_{s,r}\circ \md X^{i,n}_r$. We leave the computations as exercise.
\end{remark}

\section{Application 2: An LDP for SDEs driven by $k$-layer noises}\label{appl2}

We start recalling the notation. We fix $k$ in $\bbN$ and, for a multi-index $I$ in $\{1,\ldots n\}^k$, we use the notation $I_j$ for the $j$-th component of $I$. We denote by $X^{I,n}=X^{\{k\};I,n}$ the vector $(X^{I_1,n},\ldots X^{I_k,n})$. We take $f_j:\bbR^d\rightarrow\bbR^m$, $j=1,\ldots k$, given $C^3_b$ vector fields. We consider the following family of SDEs on $\bbR^m$ driven by $X^{i,n}$, parametrized by multi-indices $I$ in $\{1,\ldots n\}^k$:
\begin{align}
dY^{I,n}_t = \sum^k_{j=1}f_j(Y^{I,n}_t)\circ \md X^{I_j,n}_t,\ \ Y^{I,n}_0=y_0,
\end{align}
where $y_0$ is a point in $\bbR^m$ independent of $I$ and $n$ (however more general choices of initial data should be possible). We call
\begin{align*}
L_n^{Y;\{k\}} = \frac{1}{n^k}\sum_{I\in \{1,\ldots n\}^k}\delta_{Y^{I,n}};
\end{align*}
it is a random variable with values in $\cP(C^{0,\beta}([0,T];\bbR^m)$, for any $1/3<\beta<1/2$.

By rough paths theory, precisely Theorems 8.4 and 8.5 in \cite{FrHa14}, there exists a (unique) continuous function $\varphi:\cC^{0,\alpha}_g([0,T];\bbR^{kd})\rightarrow C^{0,\beta}([0,T];\bbR^m)$ such that, for every $I$ and every $n$, $Y^{I,n}=\varphi(S^{kd}(X^{I,n}))$ (actually $\varphi$ is locally Lipschitz continuous). This brings to the following LDP, as recalled in the introduction:

\begin{corro}
Fix $1/3<\beta<1/2$. The sequence $\{\text{Law}(L_n^{Y;\{k\}})|n\in\mathbb{N}\}$ satisfies a large deviation principle on $\cP(C^{0,\beta}([0,T];\bbR^m)$, endowed with the $C_0$-weak topology, with scale $n$ and good rate function given by
\begin{align*}
J^Y(Q)=\inf\{\bfJ^{\{k\}}(\mu) \mid Q=\mu\circ\varphi^{-1}\}.
\end{align*}
\end{corro}

\begin{proof}
We have
\begin{align*}
L_n^{Y;\{k\}} = \bfL_n^{\bfX;\{k\}}\circ\varphi^{-1}.
\end{align*}
as it can be easily verified by testing the two measures with a function $\psi$ in $C_b(C^{0,\beta}([0,T];\bbR^m))$. In particular $L_n^{Y;\{k\}}$ is the image of $\bfL_n^{\bfX;\{k\}}$ under the map $F:\cP(\cC^{0,\alpha}_g([0,T];\bbR^{kd})\rightarrow \cP(C^{0,\beta}([0,T];\bbR^m)$, defined by $F(Q)=Q\circ\varphi^{-1}$, which is continuous between the $C_0$-weak topologies. We then conclude by Theorem \ref{generalres} via contraction principle.
\end{proof}

\appendix

\section{Basic facts on $1$-Wasserstein metric}\label{W_section}
Let $(F,d_F)$ be a Polish space. We denote by $\cP_1(F)$ the space of probability measures on $F$ with finite first moment. It is a Polish space endowed with the $1$-Wasserstein distance $d_{W}$, namely
\begin{align}
  d_{W}(\mu, \nu)
  \;=\;
  \inf_{\pi \in \Ga(\mu, \nu)} \int_{F \times F} d_F(x^1, x^2)\; \pi(\md(x^1,x^2))
\end{align}
where $\Ga(\mu,\nu)$ is the set of all probability measures on $F \times F$ with the first marginal and the second marginal equal resp.\ to $\mu$ and $\nu$ (such measures are sometimes called transportation plans). When $F = C^{0,\al}([0,T];E)$ (for some Polish space $E$), we use the notation $d_{W,\al}$ for the $1$-Wasserstein distance associated with the $\al$-H\"older distance on $C^{0,\al}([0,T]; E)$.

We recall the following characterization of convergence in the $1$-Wasserstein metric, stated in \cite[Definition~6.8]{Vi09}. Here and in the following, we say that a map $\vp\!: F \to F'$ ($F$, $F'$ being Polish spaces) has at most linear growth if there exists $x_0 \in F$, $y_0 \in F'$ and $C \geq 0$ such that, for every $x$ in $F$
\begin{align}\label{lingrowth}
  d_{F'}(\vp(x), y_0)
  \;\leq\;
  C\,\big(1 + d_F(x,x_0)\big).
\end{align}
It is easy to see that this property is equivalent to the following fact: for any $x_0$ in $F$, $y_0$ in $F'$, there exists $C\ge0$ such that, for every $x$ in $F$, \eqref{lingrowth} holds.

\begin{lemma}\label{char_1Wasserstein}
  The following facts are equivalent:
  \begin{itemize}
  \item $\mu_n\rightarrow\mu$ in $1$-Wasserstein distance;
  \item $\int_F \vp(x)\, \mu_n(\md x) \to \int_F\vp(x)\, \mu(\md x)$ for any function continuous $\vp\!: F \to \bbR$ with at most linear growth;
  \item $\mu_n \rightharpoonup \mu$ and there exists $x_0 \in F$ such that, for any $\eta > 0$, there exists $R > 0$ verifying
    \begin{align}
      \sup_{n \geq 1} \int_{\{d(\cdot,x_0)>R\}} d(x,x_0)\, \mu_n(\md x) \;<\; \eta.
    \end{align}
  \end{itemize}
\end{lemma}
As a consequence, we have the following Corollary.
\begin{corro}\label{cont_lingrowth}
  Let $h\!: F \to F'$ be a continuous map ($F$, $F'$ being Polish spaces) with at most linear growth.  Then, the corresponding map at the level of measures, namely $\cP_1(F) \ni \mu \mapsto \mu \circ h^{-1} \in \cP_1(F')$, is continuous in the $1$-Wasserstein metric.
\end{corro}
\begin{proof}
  Using the equivalence above, it is enough to verify that, for any sequence $\{\mu_n : n \in \bbN\}$ converging to $\mu$ in $\cP_1(F)$, for any continuous function $\vp\!: F'\to \bbR$ with at most linear growth, $\int_F \vp(h(x))\, \mu_n(\md x) \to \int_F \vp(h(x))\, \mu(\md x)$.  Now, since $h$ is continuous with at most linear growth, also $\vp \circ h$ is continuous with at most linear growth, hence the convergence above holds.
\end{proof}
The following Lemma provides a wide class of compact sets in the $1$-Wasserstein metric.
\begin{lemma}\label{criterion_compactness}
  Let $G$ be a function $G\!: F \to [0, \infty]$, with compact sublevel sets and with more than linear growth.  Define the set
  \begin{align}
    K_M
    \;\ldef\;
    \Big\{\nu \in \cP_1(F) \,:\, \int_F G\, \md\nu \leq M \Big\}.
  \end{align}
  Then $K_M$ is compact (in the $1$-Wasserstein metric).
\end{lemma}
\begin{proof}
  We prove sequential compactness (which is equivalent to compactness for metric spaces). Let $\{\nu_n : n \in \bbN\}$ be a sequence of measures in $K_M$, we will prove that $\nu_n$ is tight and that there exists $x_0 \in F$ such that, for every $\eta > 0$, there exists $R > 0$ verifying
  \begin{align}\label{unifint}
    \sup_{n \geq 1} \int_{\{d(\cdot,x_0)>R\}} d(x,x_0)\, \nu_n(\md x) \;<\; \eta.
  \end{align}
  This two conditions imply the existence of a subsequence $\{\nu_{n_k} : k \in \bbN\}$ converging to some measure $\nu$ in $\cP_1(F)$ in the $1$-Wasserstein metric; it is easy to prove that $\nu$ is still in $K_M$ (since the functional $\nu\to \int_F G\, \md\nu$ is lower semi-continuous by Corollary \ref{lsc_functional}), so that $K_M$ is compact.

  For tightness, we use the compact sublevel sets property of $G$: for every $\delta > 0$, the set $\{G \leq \delta\}$ is compact and, by Markov inequality, we have, for any $n$,
  \begin{align}
    \nu_n\big[G > \de^{-1}\big]
    \;\leq\;
    \de\, \int_F G\, \md\nu_n
    \;\leq\;
    \de M.
  \end{align}
  This proves tightness.
  
  For \eqref{unifint}, we use the more than linear growth property of $G$: for some $x_0 \in F$, for any $\eta > 0$, there exists $R > 0$ such that $d(x,x_0)/G(x)<\eta$. Hence, for any $n$,
  \begin{align}
    \int_{\{d(\cdot,x_0)>R\}} d(x,x_0)\, \nu_n(\md x)
    \;\leq\;
    \eta\, \int_{\{d(\cdot,x_0)>R\}} G(x)\, \nu_n(\md x)
    \;\leq\;
    \eta M.
  \end{align}
  This proves \eqref{unifint} (up to choosing a different $R$). The lemma is proved.
\end{proof}
We conclude this section with a result on the continuity of the doubling map for measures under the $1$-Wasserstein metric. Recall that, if $(F, d)$ is a Polish space, then $(F^2, d^{\{2\}})$ is a Polish space as well, where $d^{\{2\}}((x,y),(x',y'))^2 = d(x,x')^2 + d(y,y')^2$; similarly, for any $k \geq 2$, $(F^k, d^{\{k\}})$ is a Polish space as well, where $d^{\{k\}}((x_1, \ldots, x_k),(x_1', \ldots, x_k'))^2 = d(x_1,x_1')^2 + \ldots + d(x_k,x_k')^2$.
\begin{lemma}\label{continuity_doubling}
  Let $F$ be a Polish space, $k \geq 2$ integer. Then the map
  \begin{align}
    \cP_1(F) \ni \mu \;\longmapsto\; \mu^{\otimes k} \in \cP_1(F^k)
  \end{align}
  is continuous (where $\cP_1(F,d)$, $\cP_1(F^k,d^{\{k\}})$ are endowed with the $1$-Wassestein distance induced by $d$ and $d^{\{k\}}$, respectively).
\end{lemma}

\begin{proof}
  We start with the case $k = 2$. Let $\mu$, $\nu$ be two probability measures in $\cP_1(F)$.  Let $\pi$ in $\Ga(\mu, \nu)$ be an admissible plan between $\mu$ and $\nu$, namely a probability measure on $F \times F$ with first marginal $\mu$ and second marginal $\nu$.  Then, an admissible plan $\pi^{\{2\}}$ on $F^2 \times F^2$ between $\mu \otimes \mu$ and $\nu \otimes \nu$ is built from $\Pi$ as follows: identifying $F^2 \times F^2$ with $F^4$ and calling $q_j$, $j=1, \ldots, 4$, the canonical projections, $\pi^{\{2\}}$ is the unique measure on $F^4$ such that, under $\pi^{\{2\}}$, $(q_1, q_3)$ and $(q_2, q_4)$ are i.i.d.\ with distribution $\pi$.  Indeed, with this definition, $\pi^{\{2\}} \circ (q_1,q_2)^{-1} = \pi \circ(q_1)^{-1} \otimes \pi \circ(q_2)^{-1} = \mu \otimes \mu$ and similarly $\pi^{\{2\}} \circ (q_3,q_4)^{-1} =\nu \otimes \nu$, so $\pi^{\{2\}}$ is in $\Ga(\mu \otimes \mu, \nu \otimes\nu)$. Now we have
  \begin{align*}
    d_{W}(\mu \otimes \mu, \nu \otimes \nu)
    &\;=\;
    \inf_{\xi \in \Ga(\mu \otimes \mu, \nu \otimes \nu)}
    \int_{F\times F}
      d^{\{2\}}((x^1,y^1),(x^2,y^2))\,
    \xi\big(\md((x^1,y^1),(x^2,y^2))\big)
    \\[.5ex]
    &\;\leq\;
    \inf_{\pi \in \Ga(\mu, \nu)}
    \int_{F^2 \times F^2}
      d^{\{2\}}((x^1,y^1),(x^2,y^2))\,
    \pi^{\{2\}}\big( \md((x^1,y^1),(x^2,y^2)) \big)
    \\[.5ex]
    &\;\leq\;
    \inf_{\pi \in \Ga(\mu, \nu)}
    \int_{F^2\times F^2}
      \big(d(x^1,x^2) + d(y^1,y^2)\big)\,
    \pi^{\{2\}}\big(\md((x^1,y^1),(x^2,y^2))\big)
    \\[.5ex]
    &\;\leq\;
    \inf_{\pi \in \Ga(\mu, \nu)}
    \int_{F \times F} d(x^1,x^2)\, \pi\big(\md(x^1,x^2)\big)
    \,+\, \int_{F \times F} d(y^1,y^2)\, \pi\big(\md(y^1,y^2)\big)
    \;=\;
    2d_W(\mu,\nu),
  \end{align*}
  where in the second inequality we used the simple estimate $d^{\{2\}}((x^1,y^1),(x^2,y^2)) \leq d(x^1,x^2) + d(y^1,y^2)$ and in the third inequality we used the fact that $(x^1,x^2) = (q_1,q_3)$ and $(y^1,y^2) = (q_2,q_4)$ are distributed according to $\pi$.  The estimate above implies immediately continuity (and even Lipschitz continuity) for $k = 2$.

  In the case $k = 2^h$ for some positive integer $h$, it is enough to note that $\mu \mapsto \mu^{\otimes 2^h}$ is the $h$-times iteration of the map $\mu \mapsto \mu^{\otimes 2}$. In the case $k$ general, the measure $\mu^{\otimes k}$ is obtained projecting the measure $\mu^{\otimes 2^h}$ on the first $k$ components, for some $h$ with $k \leq 2^h$, so continuity of $\mu \mapsto \mu^{\otimes k}$ follows.
\end{proof}

\section{Technical results and proofs}\label{App_C}
We start with a known result on lower semi-continuous functions, that we use at least twice in the paper.
\begin{lemma}
  Let $(E, d)$ be a metric space.  Any lower semi-continuous function $f\!: E \to(-\infty, \infty]$, bounded from below, is the pointwise supremum of an increasing sequence of continuous (actually Lipschitz) maps.
\end{lemma}
\begin{proof}
  If $f$ is identically $+\infty$, then it is enough to take $f_k \equiv k$ as Lipschitz approximants. Hence, we consider $f$ assuming at least one finite value.  We define $\{f_k : k\in \bbN\}$ as the lower envelope of $f$, namely
  \begin{align}\label{def_f_k}
    f_k(x) \;=\; \inf_{y \in E} \big\{f(y) + k\, d(x,y)\big\}.
  \end{align}
  Since $f$ is bounded from below and not identically $+\infty$, $f_k$ is a real-valued function. The sequence $f_k$ is increasing and, for every $k$, $x$, we have $f_k(x) \leq f(x)$ (by choosing $y=x$ in \eqref{def_f_k}). Moreover, for each $k$, $f_k$ is Lipschitz continuous: for every $y$, $|(f(y) + kd(x,y))-(f(y) + kd(x',y)) \leq k d(x,x')$ and therefore $|f_k(x) - f_k(x')| \leq k d(x,x')$. We are left to prove the pointwise convergence of $f_k$ to $f$.

  We start with proving convergence on the points $x$ with $f(x)$ finite. Fix $\ve > 0$ and, for every $k$, take a point $x_k$ such that $f(x_k) + k d(x,x_k) < f_k(x) + \ve$. The sequence $\{x_k : k \in \bbN\}$ converges to $x$: indeed $k d(x,x_k) \leq f_k(x) + \ve +(\inf(f))^-\le f(x)+\ve+(\inf(f))^-$ for every $k$. Therefore, by lower semi-continuity,
  \begin{align}
    f(x)
    \;\leq\;
    \liminf_{k \to \infty} f(x_k)
    \;\leq\;
    \liminf_{k \to \infty} f_k(x) + \ve.
  \end{align}
  By the arbitrariness of $\ve$, we conclude $f(x)=\lim_{k \to \infty} f_k(x)$.

  For the case $f(x) = +\infty$, fix $N>0$, by lower semi-continuity, there exists $\de > 0$ such that $f>N$ on $B(x, \de)$. Therefore $f_k(x) \geq N + k \de$ and so $\{f_k(x) : k \in \bbN\}$ converges to $+\infty = f(x)$. The proof is complete.
\end{proof}
\begin{corro}\label{lsc_functional}
  Let $(E,d)$ be a metric space and let $f\!: E \to (-\infty, \infty]$ be lower semi-continuous, bounded from below.  Then, for every sequence $\{\mu_n : n \in \bbN\}$ in $\cP(E)$, converging $C_b(E)$-weakly to $\mu$ in $\cP(E)$, it holds
  \begin{align}
    \int_E f(x)\, \mu(\md x)
    \;\leq\;
    \liminf_{n \to \infty} \int_E f(x)\, \mu_n(\md x).
  \end{align}
\end{corro}

\begin{proof}
  The previous Lemma gives that $f = \sup_{k \geq 1} f_k$, where $\{f_k : k \in \bbN\}$ is an increasing sequence of continuous functions.  We can assume, possibly replacing $f_k$ with $f_k\wedge k$, that $f_k$ is bounded for every $k$. By monotone convergence theorem, we have for every $\nu$ in $\cP(E)$
  \begin{align}
    \int_E f(x)\, \nu(\md x) \;=\; \sup_{k \geq 1} \int_E f_k(x)\, \nu(\md x).
  \end{align}
  So the function $\nu \mapsto \int_E f(x)\, \nu(\md x)$ is the supremum of a family of continuous functions in the $C_b(E)$-weak topology, therefore, by a standard argument, it is sequentially lower semi-continuous in that topology.
\end{proof}
Here is the version of Varadhan lemma we need.
\begin{theorem}[Varadhan lemma]\label{Varadhan_lemma}
  Let $E$ be a regular Haussdorff space. Suppose that $\{\mu_n : n \in \bbN\}$ is a sequence of probability measures on $E$ satisfying a large deviation principle with scale $n$ and good rate function $I$. Let $\vp\!: E \to \bbR$ be a continuous function such that
  \begin{align}\label{unif_exp}
    \limsup_{n \to \infty} \frac{1}{n}\, \log \int_E \exp(n \ga \vp)\, \md\mu_n
    \;<\;
    \infty
  \end{align}
  for some $\ga > 1$. For any $n$, let $\nu_n$ be the probability measure having density $Z_n^{-1} \me^{n\vp}$ with respect to $\mu_n$ ($Z_n$ being the normalization constant). Then the sequence $\{\nu_n : n \in  \bbN\}$ satisfies a large deviation principle with scale $n$ and rate function $J = I - \vp - \inf_E (I - \vp)$. It also holds
  \begin{align}
    \lim_{n \to \infty} \frac{1}{n}\, \log Z_n
    \;=\;
    \inf_E ( I - \vp).
  \end{align}
  In particular, if $Z_n = 1$ for each $n$ (i.e.\ if $\me^{n \vp}\,\mu_n$ is a probability measure), then $J = I - \vp$.  Furthermore, if $\{\mu_n : n\in \bbN\}$ is exponentially tight, then so is $\{\nu_n : n \in \bbN\}$ and the rate function $J$ is good.
\end{theorem}
\begin{proof}
  Apart for the last sentence, the statement is a simple consequence of Varadhan lemma in in \cite[Theorem~4.3.1, Lemma~4.3.4 and Lemma~4.3.6]{DeZe10}. The goodness of $J$ follows by \cite[Lemma!1.2.18]{DeZe10}, if we have exponential tightness for $\{\nu_n : n \in \bbN\}$.  Since $\{\mu_n : n \in \bbN\}$ is exponentially tight, for any $M > 0$, there exists $K_M$ compact set such that
  \begin{align}
    \limsup_{n \to \infty} \frac{1}{n}\log\mu_n\big[K_M^c\big] \;<\; -M.
  \end{align}
  We have
  \begin{align}
    \nu_n\big[K_M^c\big]
    \;=\;
    \frac{1}{Z_n}\,
    \int_E\, 1_{K_M^c}\, \me^{n\vp}\, \md\mu_n
    \;\leq\;
    \frac{1}{Z_n}\, \mu_n\big[K_M^c\big]^{1-1/\ga}\,
    \left( \int_E \me^{n\ga\vp}\,\md\mu_n\right)^{\!\!1/\ga}.
  \end{align}
  Now, using the assumption \eqref{unif_exp}, we easily get that
  \begin{align}
    \limsup_{n \to \infty} \frac{1}{n}\, \log\mu_n\big[K_M^c\big]
    \;<\;
    -C(M-1) - \inf_E (I - \vp)
  \end{align}
  for some constant $C>0$. The proof is complete.
\end{proof}
We prove now the lower-semi-continuity of $N_\al$.
\begin{proof}[Proof of Lemma~\ref{Nlsc}]
  Notice first that, for any $i$,
  \begin{align}
    \{N \leq i\} \;=\; \{\tau_{i+1} \geq T\}
  \end{align}
  so that lower semi-continuity of $N$ follows from upper semi-continuity of $\tau_i$, for any $i$, which we now aim to prove. We must show that, for any $i$ in $\bbN$, for any $t > 0$,
  \begin{align}
    \{\tau_{i} \geq t\}
    \;=\;
    \big\{
      X \in \cC^{0,\al}_g \,:\, \|X\|_{(1/\al)-var,[\tau_{i-1}(\bfX),t]} \leq 1
    \big\}
    \;\rdef\;
    A_i(t)
  \end{align}
  is a closed set.  We use induction on $i$. For $i=1$, since $\tau_0 = 0$, closedness follows from continuity of the $(1/\al)-var$ norm (with respect to $\bfX$). For the passage from $i$ to $i+1$, take $\{\bfX^m : m \in \bbN\}$ sequence in $A_{i+1}(t)$ converging to some $\bfX$ in $\cC_g^{0, \al}([0,T];\bbR^e)$, we must prove that $\bfX$ belongs to $A_{i+1}(t)$. By upper semi-continuity of $\tau_i$ (inductive hypothesis), we have that $\tau_i(\bfX) \geq \limsup_{m \to \infty} \tau_i(\bfX^m)$, so, for any $\de > 0$, the interval $[\tau_i(\bfX)+\delta, t]$ is contained in $[\tau_i(\bfX^m), t]$ for $m$ large enough.  So, for any $\de > 0$, by continuity and monotonicity properties of the $(1/\al)-var$ norm, we have
  \begin{align}
    \|X\|_{(1/\al)-var,[\tau_{i}(\bfX)+\delta,t]}
    \;\leq\;
    \limsup_{m \to \infty} \|X\|_{(1/\al)-var,[\tau_{i}(\bfX^m),t]} 
    \;\leq\;
    1.
  \end{align}
  By arbitrariness of $\delta>0$ and again by continuity of the norm, we get that $\|X\|_{(1/\al)-var,[\tau_{i}(\bfX),t]}\le 1$, that is $\bfX$ belongs to $A_{i+1}(t)$. The proof is complete.
\end{proof}
Now we prove Lemma \ref{Mod_regular}.
\begin{proof}[Proof of Lemma \ref{Mod_regular}]
  The Haussdorff property follows from the fact that the $(\|\cdot\|+N)^{1+\ve}$-Wasserstein topology is stronger than the $1$-Wasserstein metric (which is an Haussdorff space).

  As for the regularity property, we prove it by embedding this space into a topological vector space (which is regular). Precisely, let $V$ be the space of finite signed measures $\nu$ on $\cC_g^{0, \al}([0,T];\bbR^e)$, with finite $(|X_0|+\|\cdot\|+N)^{1+\ve}$ moment, i.e.\
  \begin{align}
    \int_{\cC^{0,\al}_g}(|X_0|+\|\bfX\|_\al+N_\al(\bfX))^{1+\ve}\, |\nu|(\md\bfX)
    \;<\;
    \infty,
  \end{align}
  where $|\nu|$ denotes the total variation measure of $\nu$. We say that a sequence $(\nu_n)_n$ converges to $\nu$ in $V$, in the $(\|\cdot\|+N)^{1+\ve}$-sense, if:
  \begin{enumerate}
  \item $\{\nu_n : n\}$ converges to $\nu$ in the $C_b$-weak topology, i.e.\ against any test function in $C_b(\cC_g^{0, \al}([0,T];\bbR^e))$;
  \item we have
    \begin{align}
      \sup_{n \geq 1}
      \int_{\cC^{0,\al}_g}(|X_0|+\|\bfX\|_\al+N_\al(\bfX))^{1+\ve}\,|\nu_n|(\md\bfX)
      \;<\;
      \infty.
    \end{align}
  \end{enumerate}
  This defines a topology on $V$ which we call $(\|\cdot\|+N)^{1+\ve}$ signed topology (or just signed topology). It is easy to see that this topology is Haussdorff and that it makes the operations $V\times V\ni (\nu_1,\nu_2)\mapsto \nu_1+\nu_2 \in V$, $\bbR\times V\ni (\alpha,\nu)\mapsto \alpha\nu \in V$ continuous; so $V$ is a Haussdorff topological vector space with the $(\|\cdot\|+N)^{1+\ve}$ signed topology. As a general result in topology, any Haussdorff topological vector space is regular, so $V$ is regular.

  It is also easy to see that $\cP_{(\|\cdot\|+N)^{1+\ve}}(\cC^{0,\al}_g)$ is closed in $V$ and that the $(\|\cdot\|+N)^{1+\ve}$ signed topology induces the $(\|\cdot\|+N)^{1+\ve}$-Wasserstein topology on $\cP_{(\|\cdot\|+N)^{1+\ve}}(\cC^{0,\al}_g)$: any subset in $\cP_{(\|\cdot\|+N)^{1+\ve}}(\cC^{0,\al}_g)$ which is closed in the $(\|\cdot\|+N)^{1+\ve}$-Wasserstein topology is also closed in the signed topology and, viceversa, the intersection of any closed (in the signed topology) subset of $V$ with $\cP_{(\|\cdot\|+N)^{1+\ve}}(\cC^{0,\al}_g)$ is closed in the $(\|\cdot\|+N)^{1+\ve}$-Wasserstein topology.

  This allows to prove that $\cP_{(\|\cdot\|+N)^{1+\ve}}(\cC^{0,\al}_g)$ is regular (with the original $(\|\cdot\|+N)^{1+\ve}$-Wasserstein topology). Indeed, let $\mu$ be in $\cP_{(\|\cdot\|+N)^{1+\ve}}(\cC^{0,\al}_g)$ and let $C$ be a closed set in the $(\|\cdot\|+N)^{1+\ve}$-Wasserstein topology. Since $C$ is closed also in the signed topology, then there exist $A$, $B$ disjoint subset of $V$, open in the signed topology, such that $\mu\in A$ and $C\subseteq B$. Hence, calling $A'$, resp.\ $B'$ the intersection of $A$, resp.\ $B$, with $\cP_{(\|\cdot\|+N)^{1+\ve}}(\cC^{0,\al}_g)$, then $A'$, $B'$ are two disjoint subset of $\cP_{(\|\cdot\|+N)^{1+\ve}}(\cC^{0,\al}_g)$, open in the $(\|\cdot\|+N)^{1+\ve}$-Wasserstein topology, with $\mu\in A'$ and $C\subseteq B'$. This proves regularity of $\cP_{(\|\cdot\|+N)^{1+\ve}}(\cC^{0,\al}_g)$. The proof is complete.
\end{proof}
Here we prove that the enhanced empirical measure associated with a rough path in $\bbR^{nd}$ is a continuous function (in the modified Wasserstein topology) of the rough path itself.
\begin{lemma}\label{lemma:RP_cont_empmeas}
  Fix $n$ and $k$ (with $n \geq k$). The map $G_n\!:\cC^{0,\al}_g([0,T]; \bbR^{nd}) \to \cP_{(\|\cdot\|_\al+N_\al)^{1+\ep}}(\cC^{0,\al}_g([0,T];\bbR^{kd}))$ given by
  \begin{align*}
    G_n(\bfX)
    \;=\;
    \bfL^{\bfX,\{k\}}_n
    \;=\;
    \frac{1}{n^k}\, \sum^n_{i_1, \ldots, i_k=1} \de_{\bfX^{\{k\};i_1, \ldots, i_k}}
  \end{align*}
  is continuous (in particular measurable).
\end{lemma}
\begin{proof}
  Let $\{\bfX^m : m \in \bbN\}$ be a sequence of $nd$-dimensional geometric rough paths, converging to $\bfX$ in $\cC^{0,\al}_g([0,T];\bbR^{nd})$ (as $m\rightarrow+\infty$), we have to prove that $\bfL^{\bfX^m,\{k\}}_n$ converges to $\bfL^{\bfX,\{k\}}_n$ in the modified Wasserstein topology.
  We start proving convergence in the $C_b$-weak topology. For any $\varphi$ in $C_b(\cC^{0,\al}_g([0,T];\bbR^{kd}))$, we have
  \begin{align*}
    \int_{\cC^{0,\al}_g([0,T];\bbR^{kd})} \vp\, \md\bfL^{\bfX^m,\{k\}}_n
    \;=\;
    \frac{1}{n^k}\, \sum^n_{i_1, \ldots, i_k=1} \vp(\bfX^{m,\{k\}; i_1, \ldots, i_k}),
  \end{align*}
  so convergence of $\int \vp\, \md\bfL^{\bfX^m,\{k\}}_n$ to $\int \vp\, \md\bfL^{\bfX,\{k\}}_n$ follows from continuity of $\varphi$ (and of the projections on the $(i_1,\ldots i_k)$ components).

  To conclude, we have to prove that
  \begin{align*}
    \sup_{m \geq 1}
    \int_{\cC^{0,\al}_g([0,T];\bbR^{kd})}
      \big(|Y_0|+\|Y\|_\al+N_\al(Y)\big)^{1+\ep}\,
    \bfL^{\bfX^m,\{k\}}_n(\md\bfY)
    \;<\;
    \infty.
  \end{align*}
  For this, we remind that, for any geometric rough path $Y$, $N_\al(\bfY)\le \|\bfY\|_{(1/\al)-\mathrm{var},[0,T]}^{1/\al}$ (see for example \cite{FrHa14}, Section 11.2.3) and $\|\bfY\|_{(1/\al)-\mathrm{var},[0,T]}\le C\|\bfY\|_\al$ for some constant $C$ (as easily verified). Therefore (using also that $\|\bfY^{\{k\};i_1,\ldots i_k}\|_\al \leq C\|\bfY\|_\al$ for some $C$), we get
  \begin{align*}
    &\int_{\cC^{0,\al}_g([0,T];\bbR^{kd})}
      \big(|Y_0|+\|Y\|_\al+N_\al(Y)\big)^{1+\ep}\,
    \bfL^{\bfX^m,\{k\}}_n(\md\bfY)
    \\[.5ex]
    &\mspace{36mu}=\;
    \frac{1}{n^k}\, \sum^n_{i_1, \ldots, i_k=1}
    \big(
      |X^{m,\{k\};i_1, \ldots, i_k}_0| + \|\bfX^{m,\{k\};i_1, \ldots, i_k}\|_\al
      + N_\al(\bfX^{m,\{k\};i_1,\ldots i_k})
    \big)^{1+\ep}
    \\[.5ex]
    &\mspace{36mu}\leq\;
    C\big(|X^m_0|+\|\bfX^m\|_\al+\|\bfX^m\|_\al^{1/\al}\big)^{1+\ep}
  \end{align*}
  and the RHS above is uniformly bounded in $m$. The proof is complete.
\end{proof}
Finally we prove Lemma~\ref{lemma:Strat_approx}, starting from Corollary 13.22 in \cite{FrVi10}, following Exercise 13.22 there, and Lemma \ref{N_Gaussian}, starting from Theorems 11.9 and 11.13 in \cite{FrHa14} (see also \cite{CLL13}, Theorem 6.3).
\begin{proof}[Proof of Lemma~\ref{lemma:Strat_approx}]
  Corollary 13.22 in \cite{FrVi10} applies clearly also to Brownian rough path starting from any initial measure (since $\bfB$ and $\bfB^{(m)}$ start from the same point) and gives the existence of a constant $C>0$ such that, for every $q\ge1$, for every $m$, it holds
  \begin{align*}
    \mean\!\big[d_\al(\bfB^{(m)},\bfB)^q\big]
    \;\leq\;
    (Cq^{1/2}m^{-\eta/2})^q.
  \end{align*}
  From this we get the following estimate on the exponential of the distance above: for any $\rho > 0$,
  \begin{align*}
    \mean\!\big[\exp(\rho d_\al(\bfB^{(m)},\bfB))\big]
    \;=\;
    1 + \sum^\infty_{q=1}
    \frac{\rho^q\,\mean\!\big[d_\al(\bfB^{(m)},\bfB)^q\big]}{q!}
    \;\leq\;
    1 + \sum^\infty_{q=1} \frac{(\rho C q^{1/2}m^{-\eta/2})^q}{q!}.
  \end{align*}
  Using the elementary estimate $q^q\le e^{q-1}q!$ (which can be easily proved by induction on $q$), we have
  \begin{align*}
    \mean\!\big[\exp(\rho d_\al(\bfB^{(m)},\bfB))\big]
    \;\leq\;
    1 + \sum^\infty_{q=1}(\me \rho Cm^{-\eta/2})^q.
  \end{align*}
  So, taking $\rho = m^{\eta/2}/(2\me C)$, we get that this series converges. Hence,
  \begin{align}
    \mean\!\big[\exp\big((2\me C)^{-1} m^{\eta/2} d_\al(\bfB^{(m)},\bfB)\big)\big]
    \;<\;
    \infty.
  \end{align}
The proof is complete. [Notice that some estimates were not optimal: in fact the result holds also for $d_\al(\bfB^{(m)},\bfB)^2$ replacing $d_\al(\bfB^{(m)},\bfB)$.]
\end{proof}
\begin{proof}[Proof of Lemma \ref{N_Gaussian}]
  Notice that (for $\ve<1$, using independence of the initial datum and the increments of Brownian motion)
  \begin{align}
    \mean\!\big[\exp\big(c(|B_0|+\|\bfB\|_\beta+N_\al(\bfB))^{1+\ve}\big)\big]
    \;\leq\;
    \mean\!\big[\me^{2c|B_0|}\big]\,
    \mean\!\big[\exp\big(2c(\|\bfB\|_\beta+N_\al(\bfB))^{1+\ve}\big)\big].
  \end{align}
  Now, $\mean\!\big[\exp\big(2c(\|\bfB\|_\beta + N_\al(\bfB))^{1+\ve}\big)\big]$ is finite (actually for every $c>0$), as proved in Theorems 11.9 and 11.13 in \cite{FrHa14}; $\bbE\!\big[\me^{2c|B_0|}\big] = \int_{\bbR^e} \me^{2cx}\, \tilde{\la}(\md x)$ is finite because of the exponential integrability condition \ref{eq:exp:intcond} (replacing $c$ with $2c$). The same proof applies to $\bfB^{11}$ (and to $\bfB^{\{k\};i_1,\ldots, i_k}$ for any multi-index $(i_1,\ldots, i_k)$ also with repetition of indices).
\end{proof}

\section*{Acknowledgements}

P.K. Friz's and M. Maurelli's research was supported by the Research Center MATHEON through project C-SE8, funded by the Einstein Center for Mathematics Berlin, and by ERC grant agreement nr. 258237, under the European Union's Seventh Framework Programme (FP7/2007-2013).


\end{document}